\documentclass[12pt,reqno]{amsart}
\usepackage{fancyhdr}
\usepackage{geometry}
\usepackage{amsmath, amsthm, amssymb, bm, color, framed, graphicx, hyperref, mathrsfs}
\usepackage{cite}
\usepackage{tikz}
\usetikzlibrary{hobby}
\usepackage{tikz-3dplot}
\usepackage{pgfplots}

\usepackage[all]{xy}
\usepackage{bbm}
\usepackage{cancel}
\usepackage{esint}
\usepackage{float}
\definecolor{shadecolor}{RGB}{241, 241, 255}
\newtheorem{theorem}{Theorem}

\def\pa{\partial}

\def \M {\mathcal{M}}

\def\be{\begin{equation}}
\def\ee{\end{equation}}
\def\befr{\begin{frame}}
\def\enfr{\end{frame}}

\def\M{{\bf M}}

\def\De{\Delta}

\def\Ric{\operatorname{Ric}}
\def\Vol{\operatorname{Vol}}
\newtheorem{lemma}{Lemma}
\newtheorem{corollary}{Corollary}
\newtheorem{remark}{Remark}
\newtheorem{definition}{Definition}
\numberwithin{theorem}{section}
\numberwithin{equation}{section}
\numberwithin{lemma}{section}
\numberwithin{corollary}{section}
\numberwithin{remark}{section}
\numberwithin{definition}{section}

\begin{document}
\title{On the space of harmonic forms and Schr\"odinger heat equations on noncompact manifolds}
\author{Jiuyuan Cui}
\address{School of Mathematical Science, Fudan University, Shanghai 200433, China}
\email{jycui23@m.fudan.edu.cn}
\date{}
\maketitle

\begin{abstract}
It is known that on complete Riemannian manifolds with nonnegative Ricci curvature, the space of harmonic 1-forms of polynomial growth is finite-dimensional. 
In this paper, we extend this finiteness phenomenon to some manifolds whose Ricci curvature is allowed to be negative. 
More precisely, under a scale-invariant $L^2$-Sobolev inequality, a weak Hardy inequality, and quadratic Ricci curvature decay with a sufficiently small coefficient, we prove that the space of polynomial-growth harmonic 1-forms is finite-dimensional. 
We also construct a counterexample showing the necessity of the smallness condition.

Our result is closely related to an earlier theorem of L.F.Tam \cite{tam1998note}, where the finite-dimensionality of polynomial-growth harmonic 1-forms was established under a faster curvature decay assumption, without any smallness condition on the coefficient. 
In comparison, our result treats the critical quadratic decay regime, at the expense of requiring the corresponding coefficient to be sufficiently small. 
As a byproduct, we obtain an analogous finite-dimensionality result for polynomial-growth solutions to the Schrödinger heat equation, which generalizes the finite-dimensionality results in \cite{lin2019ancient},\cite{colding2021optimal} and \cite{lin2026space} to some extent.

\end{abstract}
\section{Introduction}
 Let $(\M,g)$ be an $n$-dimensional complete Riemannian manifold. 
 Given a positive integer $l$, let $\mathcal H_{l}(\M)$ denote the linear space of harmonic functions with polynomial growth of order $l$. 
 Namely, $u\in\mathcal H_{l}(\M)$ if $\Delta_gu=0$ and for some $p\in\M$ and a constant $C_u$ depending on $u$ such that$$\sup\limits_{B_p(R)}|u(x)|\leq C_u(1+R)^{l}.$$
 S.T.Yau conjectured that this linear space is finite dimensional for each $l$ when $\Ric\geq 0$. 
 This conjecture was proved by Colding and Minicozzi in \cite{colding1997harmonic} under a much weaker assumption. 
 They proved the following remarkable result:\begin{theorem} [Colding, Minicozzi]
    
    If  $\M$ is an open manifold which has the doubling property and satisfies a uniform Neumann-Poincar\'e inequality, 
    then for all $l>0$, there exists $C>0$ depending on $l$ and constants in volume doubling property and Neumann-Poincar\'e inequality, 
    such that $\dim\mathcal{H}_l(\M)\leq C.$
\end{theorem}
There have been several important generalizations of this result. 
Colding and Minicozzi \cite{colding1998liouville} extended it to harmonic sections of Hermitian vector bundles, 
following an earlier announcement in \cite{colding1996generalized}. 
Shortly thereafter, Li \cite{li1997harmonic} provided a shorter proof. 
More precisely, he obtained a sharp dimension estimate for the space of sections of a vector bundle with polynomial growth whose squared norms are subharmonic. 
This result was later generalized by Hua \cite{hua2011harmonic} to the setting of Alexandrov spaces using a similar argument. 
Similar results have been obtained for complete manifolds under weaker curvature assumptions, without requiring nonnegative curvature. 
In 1998, Tam \cite{tam1998note} studied harmonic forms on complete Riemannian manifolds under the assumption of rapidly decaying curvature. 
He showed that if the Ricci curvature decays faster than $r^{-n}$ and a volume comparison condition holds at infinity, then the space of harmonic 1-forms with polynomial growth is finite dimensional. 
Similar results for higher-order harmonic forms are obtained by replacing the Ricci curvature assumption with a corresponding decay assumption on the curvature operator. 
Recently, J.Y.Wu and P.Wu\cite{wu2023harmonic} gave an estimate of $\dim\mathcal{H}_l(\M)$ when $\M$ is a complete shrinker and the scalar curvature of $\M$ has quadratic decay. 

In parallel, extensive studies of related problems have been carried out in complex geometry, particularly in the K\"ahler setting, where a rich theory has been developed using methods from complex analysis and differential geometry. 
We now briefly recall some representative results in this direction. 
Let $\M$ be a complete K\"ahler manifold and $\mathcal{O}_l(\M)$ be the linear space of holomorphic functions with polynomial growth of order $l$. 
Ni \cite{ni2004monotonicity} discovered a monotonicity formula and proved that if $\M$ is a complete K\"ahler manifold with nonnegative bisectional curvature, then $\dim\mathcal{O}_l(\M)\leq C(n)l^n$. 
Then in \cite{liu2016three}, Liu proved that $\M$ satisfies three-circle theorem if and only if the holomorphic sectional curvature is nonnegative. 
As an application, Liu extended Ni's result to the case of nonnegative holomorphic sectional curvature. 
In 2021, Liu \cite{liu2021dimension} proved that if $\M$ has nonnegative Ricci curvature and Euclidean volume growth, i.e., $\Vol(B_p(r))>\alpha r^{2n}$ for all $r$, then $\dim\mathcal{O}_l(\M)\leq C(n,\alpha)l^n$, see also \cite{liu2023brief}. 
In addition, the case of K\"ahler Ricci shrinking soliton was studied by He and Ou, see \cite{he2024dimension,he2025dimension}. 
In the K\"ahler case, there are some results showing that $\dim\mathcal{O}_l(\M)$ is less than or equal to $\dim\mathcal{O}_l(\mathbb{C}^n).$ 
For example, Ni \cite{ni2004monotonicity} showed that if a complete K\"ahler manifold $\M$ with nonnegative bisectional curvature has maximal volume growth, then $\dim\mathcal{O}_l(\M)\leq \dim\mathcal{O}_l(\mathbb{C}^n)$.
Shortly after, in \cite{chen2006sharp}, Chen-Fu-Le-Zhu removed the volume condition by using a similar approach as in \cite{ni2004monotonicity}. 
Another result was given by Liu \cite{liu2016three}. 
He showed that if a complete K\"ahler manifold has nonnegative holomorphic sectional curvature, then $\dim\mathcal{O}_l(\M)\leq \dim\mathcal{O}_l(\mathbb{C}^n)$. 
Thus, one may ask whether this holds in the Riemannian case. However, this fails in general in Riemannian setting, see \cite{donnelly2001harmonic}. 

These dimension estimates have various applications. 
For instance, Kleiner \cite{kleiner2010new} gave a new proof of Gromov's theorem on groups of polynomial growth, originally due to Gromov \cite{gromov1981groups}. 
Recently, Lin et al. \cite{lin2026space} investigated the space of harmonic functions with polynomial growth on connected sums of manifolds. 
They proved that this space is finite dimensional under fairly general geometric assumptions, without requiring the volume doubling property. 
As a consequence, they eliminated the need for any local topological or geometric conditions near a reference point.

Since harmonic functions arise as stationary solutions to the heat equation, it is natural to extend the above discussion to the time-dependent setting. 
It is well known that the study of the heat equation is a fundamental topic in geometric analysis. 
One important problem in this area is to investigate the properties of ancient solutions, which play a crucial role in singularity analysis. 
We say that $u$ is an ancient solution of heat equation if $$\pa_t u-\Delta u=0,\quad (t,x)\in (-\infty,0)\times\M.$$ 
Let $\mathcal{P}_l(\M)$ be the linear space of ancient caloric functions with polynomial growth of order $l$, 
i.e., $u\in\mathcal{P}_l(\M)$ if $u$ is an ancient solution of the heat equation and for some $p\in\M$ and constant $C_u$ depending on $u$, we have $$\sup\limits_{[-R^2,0]\times B_p(R)}|u|\leq C_u(1+R)^l.$$
In 2019, Lin and Zhang \cite{lin2019ancient} proved that ancient caloric functions with polynomial growth are necessarily polynomials in the time variable, 
and consequently obtained a dimension estimate for the space of such functions of growth order $l$. 
This estimate was later refined by Colding and Minicozzi \cite{colding2021optimal}, who established a sharp bound. 
They showed that $$\dim\mathcal{P}_{2k}(\M)\leq\sum\limits_{i=0}^k\dim\mathcal{H}_{2i}(\M),$$ if $\M$ has polynomial volume growth. 
As a consequence, $\dim\mathcal{P}_l(\M)\leq Cl^{n-1}$ if Ric$\geq0$. 
This result was further generalized by Ho \cite{ho2023ancient} to the setting of pseudohermitian manifolds using techniques inspired by \cite{colding2021optimal}. 
Analogous dimension estimates have also been established in the setting where the geometry evolves by a geometric flow. 
In particular, for ancient mean curvature flows, Maria Calle proved in 2006 \cite{calle2006bounding} that for an ancient solution $(\M_t,g_t)$ satisfying the growth condition $$1\leq\frac{\mathcal{A}(\M\cap E_r)}{r^n}\leq V_{\M}<\infty,$$
one has the dimension estimate $\dim\mathcal{H}_l(\M)\leq C(n)V_\M l^{n-1}$. 
We refer to \cite[page 657]{calle2006bounding} or \cite[page 17]{calle2007mean} for the definition of $\mathcal{A}(\M\cap E_r)$ and further discussions. 
More recently, a sharp bound for $\dim\mathcal{H}_l(\M)$ in the context of ancient mean curvature flows with bounded entropy was later obtained by Colding and Minicozzi \cite{colding2020complexity}. 
Among its applications, they showed that the codimension of an ancient mean curvature flow in Euclidean space is bounded by its entropy. 

While the above works focus primarily on scalar caloric functions, many problems in geometric analysis naturally involve differential forms. 
Since differential forms generalize functions and encode richer geometric information, it is natural to investigate whether the polynomial time structure and finite-dimensionality results extend to this setting. 
In 2006, Chen and Sung \cite{chen2006dimension} proved that, under suitable curvature assumptions—namely, nonnegative Ricci curvature for differential 1-forms and nonnegative curvature operator for higher-order forms—the space of differential forms with polynomial growth on a metric $g'$ uniformly equivalent to $g$ is finite dimensional. 
For caloric differential forms, Sung and Tsai \cite{sung2022polynomial} showed that on manifolds with polynomial volume growth, ancient caloric differential forms admit a representation as polynomials in the time variable with differential form coefficients. 
Moreover, by adapting the methods of Colding and Minicozzi \cite{colding2021optimal}, they established the finite dimensionality of the space of caloric differential forms with polynomial growth. 

Before stating our main theorem, we introduce some notation that will be used throughout the paper. 
We denote the underlying manifold by $\M$. 
For $x\in \M$ and $R>0$, $B_x(R)$ denotes the geodesic ball centered at $x$ with radius $R$. 
For $0<s<t$, we write $A(x,s,t):=B_x(t)\setminus \overline{B_x(s)}$ for the annulus centered at $x$ with inner radius $s$ and outer radius $t$. 
For any measurable subset $\Omega\subseteq \M$, $\Vol(\Omega)$ denotes its Riemannian volume. 
Finally, after fixing a base point $p\in \M$, we write $r(x):=d(p,x).$

Our first main result can be stated as follows.

\begin{theorem}\label{thm1}
Let $(\M^n,g),n\geq 3$ be a complete $n$-dimensional Riemannian manifold satisfying the following assumptions:
\begin{enumerate}
    \item [(a)] $L^2$-Sobolev inequality: there exists a constant $C_S>0$ such that for any geodesic ball $B_x(R)$ centered at $x\in\M$ with radius $R$ and for any $u\in C_0^\infty(B_x(R))$, we have \begin{equation}
    \bigg(\fint_{B_x(R)}|u|^{\frac{2n}{n-2}}dx\bigg)^{\frac{n-2}{n}}\leq C_S R^2\bigg(\fint_{B_x(R)}|\nabla u|^2dx\bigg);
\end{equation}
    \item [(b)] The weak Hardy type inequality: there exists a constant $C_H$ such that the following inequality holds for all $u\in C^\infty_0(\M)$:
    $$\int_\M \frac{u^2}{1+r^2(x)}dx\leq C_H\int_\M|\nabla u|^2dx;$$
    \item [(c)] The Ricci curvature of $\M$ satisfies, for some $0<\varepsilon<C_H^{-1}$, $$\Ric_g\geq -\frac{\varepsilon}{1+r^2(x)}g.$$ 
\end{enumerate}
Then for every $l>0$, the space of harmonic $1$-forms of polynomial growth of order $l$ is finite-dimensional.
\end{theorem}
We now comment on the assumptions in Theorem \ref{thm1}.

\begin{remark}
    Although our main theorem involves several assumptions, the class of manifolds satisfying these conditions is nonempty and, in fact, quite broad. 
    For instance, such manifolds can be constructed by suitably perturbing Euclidean space; see section \ref{subsec5.1}. 
    More generally, $\operatorname{SOB}(n)$-manifolds defined below satisfy all the assumptions of our theorem, by taking $\alpha=1$ and $\alpha=\frac{n}{n-2}$ in \cite[Theorem 1.2]{hein2011weighted}, respectively, except for the smallness condition on the curvature coefficient $\varepsilon$, where $n$ is the dimension of the manifold.
    On the other hand, in section \ref{subsec5.2}, we construct a counterexample to show that the smallness assumption is necessary.
\end{remark}
The notion of an $\operatorname{SOB}(\beta)$-manifold was introduced by Hein in \cite{hein2011weighted}. 
In Hein's original definition, every sufficiently large annulus is required to be connected. 
This connectivity condition was later weakened by Li in his study of Calabi--Yau metrics; see \cite[Definition 5.7]{li2019new}. 
Instead of requiring all sufficiently large annuli to be connected, Li imposes an annular connectivity condition for points on sufficiently large geodesic spheres.

For the purpose of establishing Hein's weighted Sobolev inequality, an even weaker connectivity condition is sufficient. 
Namely, it is enough to have a sequence of connected annuli at uniformly controlled geometric scales whose union, together with a fixed geodesic ball, covers the whole manifold. 
We therefore use the following weaker variant of the $\operatorname{SOB}(\beta)$ condition.

\begin{definition}\label{weakSOB}
    Let $\beta>0$ be a fixed real number. 
    A complete manifold $\M$ is called $\operatorname{SOB}(\beta)$ if there exists $p\in\M$ and $C> 1$ such that:

    (i) there exists $\eta>1,r_0>0$ such that $\M=B_p(r_0)\cup\bigcup\limits_{k=1}^\infty A(p,\eta^{k-1}r_0,\eta^kr_0)$, where $B_p(r_0)$ and $A(p,\eta^{k-1}r_0,\eta^kr_0)$ are connected;

    (ii) for all $r\geq C$, $\Vol(B_p(r))\leq Cr^\beta$;

    (iii) for all $x\in\M$ such that $r(x)\geq C$, we have $\Vol(B_x((1-\frac{1}{C})r(x)))\geq \frac{1}{C}r(x)^\beta$ and $\Ric\geq-Cr(x)^{-2}$.
\end{definition}

Under this weaker connectivity assumption, Hein's weighted Sobolev inequality still follows from the original argument. 
Indeed, an inspection of the proof of \cite[Theorem~1.2]{hein2011weighted} shows that the connectivity assumption is used only through the existence of a sequence of connected annuli with uniformly controlled scales. 

For broader applications, we further analyze our assumptions line by line.
\begin{remark}
The scale-invariant $L^2$-Sobolev inequality is a standard analytic assumption in geometric analysis and is satisfied by many familiar classes of manifolds. 
For instance, it holds on manifolds with nonnegative Ricci curvature and Euclidean volume growth; see \cite[Section 3.3.5]{saloff2002aspects}. 
Moreover, combined with the Caccioppoli inequality and Moser iteration, the scale-invariant Sobolev inequality yields a mean value inequality for nonnegative subharmonic functions. 
Such mean value estimates play a fundamental role in dimension estimates for spaces of harmonic functions, harmonic forms, and other solutions of elliptic equations with prescribed growth. 
Thus, the scale-invariant Sobolev inequality is a natural and standard hypothesis in this type of finiteness problem.
\end{remark}
\begin{remark}
The weak Hardy inequality is another important analytic assumption in Theorem \ref{thm1}. 
In our argument, the negative part of the Ricci curvature appears as a potential term with the critical decay rate $O(r^{-2})$. 
The weak Hardy inequality allows us to control this term and, under the smallness assumption on the curvature coefficient, to absorb its contribution into the main energy term.

There is a broad class of manifolds for which such a weak Hardy inequality is available. 
For instance, examples can be obtained by suitably perturbing Euclidean space; see Section \ref{subsec5.1}. 
More generally, it follows from Hein's weighted Sobolev inequality on $\operatorname{SOB}(\beta)$-manifolds, with $\beta>2$, by taking $\alpha=1$ in \cite[Theorem~1.2]{hein2011weighted}.

\end{remark}
\begin{remark}
Our result should also be compared with earlier work on curvature decay. 
We emphasize that Theorem \ref{thm1} allows the Ricci curvature to be negative, thereby relaxing the nonnegativity assumption appearing in classical finiteness results.

An earlier theorem of L.F.Tam \cite{tam1998note} states that if a complete manifold $\M$ satisfies a suitable volume comparison condition and
$\Ric\geq-\lambda(r)g,$ 
where $$\int_0^\infty r^{n-1}\lambda(r)\,dr<\infty,$$ then, for each $l>0$, the space of harmonic $1$-forms of polynomial growth of order at most $l$ is finite-dimensional, with dimension bounded above by $Cl^{n-1}$. 
Analogous results for differential forms of higher degree were also established there. 
Tam's result assumes a faster curvature decay without requiring a smallness condition on the coefficient, whereas our theorem treats the quadratic decay regime at the expense of imposing such a smallness condition. 
In this sense, the two results are complementary.

The quadratic scale also appears to be a natural borderline from the viewpoint of curvature. Indeed, a classical result of H.Choi \cite{choi1984asymptotic} shows that if the radial curvature $K_{\mathrm{rad}}$ of a rotationally symmetric manifold satisfies
$$K_{\mathrm{rad}}\leq-\frac{A}{r^2\log r},\qquad A>1,$$
then the asymptotic Dirichlet problem is solvable. 
In particular, such manifolds admit an infinite-dimensional space of bounded harmonic functions. 
Since $du$ is a harmonic $1$-form whenever $u$ is harmonic, standard gradient estimates also lead to an abundance of bounded harmonic $1$-forms. 
This provides further evidence that the quadratic curvature scale is a natural borderline in finite-dimensionality problems for harmonic functions and forms.
\end{remark}

The role of the weak Hardy inequality in our proof can now be summarized as follows. 
The negative part of the Ricci curvature contributes a critical $r^{-2}$-type potential term in the energy estimates. 
The Hardy inequality provides precisely the control needed for this term, while the smallness of its coefficient allows it to be absorbed into the left-hand side. 
This mechanism is one of the main analytic ingredients in the proof of Theorem \ref{thm1}.

It is well known that the classical Hardy inequality in Euclidean space $\mathbb{R}^n$ asserts that for $n \ge 3$,
$$\int_{\mathbb{R}^n} \frac{|\phi|^2}{|x|^2}\, dx
\le \left( \frac{2}{n-2} \right)^2
\int_{\mathbb{R}^n} |\nabla \phi|^2\, dx$$
for all $\phi \in C_0^\infty(\mathbb{R}^n \setminus \{0\})$, and the constant $\left( \frac{2}{n-2} \right)^2$ is sharp.
Various analogues of this inequality have been established on Riemannian manifolds. 
In particular, Carron \cite{carron1997inegalites} proved weighted $L^2$-Hardy inequalities under suitable assumptions on a weight function $\rho$, including
$$ |\nabla\rho|=1, \quad \Delta\rho\geq \frac{C}{\rho}. $$
Kombe and \"Ozaydin \cite{kombe2009improved} later obtained related and more general Hardy-type inequalities. 
Taking $\rho=r(x)$, these results imply, for instance, that if
$$ \Delta r\geq \frac{\gamma}{r}, \quad \gamma>1, $$
then
$$ \int_\M \frac{|u|^2}{r^2}\,dx \leq \left(\frac{2}{\gamma-1}\right)^2 \int_\M |\nabla u|^2\,dx. $$
However, requiring such a positive global lower bound for $\Delta r$ is quite restrictive and is not suitable for our setting. 
Indeed, the distance function fails to be smooth on the cut locus, and G\'en\'erau \cite{generau2020laplacian} showed that $\Delta r=-\infty$ there in the barrier sense. 
This motivates the weaker Hardy-type assumption adopted in this paper.

By Lemma \ref{volume_growth}, every manifold satisfying the assumptions of Theorem \ref{thm1} has polynomial volume growth. The following corollary follows immediately from the result of Sung-Tsai \cite{sung2022polynomial}.

\begin{corollary}
Let $\M$ be a complete manifold as in Theorem \ref{thm1}. Then for every $l>0$, the space of ancient caloric $1$-forms with polynomial growth of order $l$ is finite dimensional.
\end{corollary}

Our main theorems are based on the Bochner formula, in which the curvature term appears as a potential term. From this point of view, the finite-dimensionality problem for harmonic forms is closely related to the study of Schrödinger-type operators with decaying potentials. The analytic core of the argument is governed by the behavior of the potential term rather than its geometric origin.

Motivated by this observation, we study stationary Schrödinger equations of the form $(\Delta + V)u = 0,$ where the potential $V$ satisfies assumptions analogous to those imposed on the curvature term in the geometric setting. By adapting the same strategy as in the proof of Theorem~\ref{thm1}, we obtain finite-dimensionality results for spaces of polynomial growth solutions.

\begin{theorem}\label{dim_Sch}
Let $(\M,g)$ be a complete $n$-dimensional Riemannian manifold satisfying assumptions (a)(b)(c) as in Theorem \ref{thm1}. And let
$$\mathcal{S}_l(\M):=\left\{u \in C^\infty(\M)\; \middle|\;\Delta u + V u = 0,\quad|u(x)| = O(r(x)^l)\ \text{as } r(x)\to\infty\right\},$$ where $V(x)$ is a smooth function such that $$0\leq V(x)\leq \frac{\varepsilon_0}{1+r^2(x)},$$and $\varepsilon_0$ is sufficiently small such that $\varepsilon_0 C_H<1$. Then $\dim\mathcal{S}_l(\M)<\infty$ for every $l>0$.
\end{theorem}

Since the argument is fundamentally energy-based, it also extends to the parabolic setting. More precisely, for the Schrödinger heat equation $(\partial_t - \Delta - V)u = 0,$ we derive a Caccioppoli-type inequality and obtain corresponding finite-dimensionality results for ancient solutions of polynomial growth.

\begin{theorem}\label{polyrepre}
Fix $l>0$. Let $(\M,g)$ be a complete $n$-dimensional Riemannian manifold with polynomial volume growth. Moreover, suppose $\M$ satisfies the assumption (b) in Theorem \ref{thm1}. 
And let $$\mathcal{SH}_l(\M)
:= \left\{u\in C^\infty((-\infty,0)\times\M)\middle|\,
\partial_tu-\Delta u-Vu=0, \quad
\sup_{[-R^2,0]\times B_p(R)}|u|\le C(1+R)^l\right\},$$
be the space of ancient solutions to the Schr\"odinger heat equation with polynomial growth of order $l$, where $V$ is a smooth function such that $$0\leq V(x)\leq\frac{\varepsilon_0}{1+r^2(x)},$$ where $\varepsilon_0$ is sufficiently small such that $\varepsilon_0 C_H<1$.
 Then every $u\in\mathcal{SH}_l(\M)$ can be expressed as a polynomial in time $t$.
\end{theorem}

Using a standard linear algebra argument (see section~\ref{sec7}), we further obtain the finiteness of the solution space $\mathcal{SH}_l(\M)$.

\begin{theorem}\label{dim_est_anc}
Let $\M$ be a complete manifold satisfying assumptions (a)(b)(c) as in Theorem \ref{thm1} and $\mathcal{SH}_l(\M)$ be as in Theorem \ref{polyrepre}. Then $\dim\mathcal{SH}_l(\M)<\infty$ for every $l>0$.
\end{theorem}
This paper is organized as follows. 
Section~\ref{sec2} introduces some basic preliminaries. 
Section~\ref{sec3} uses an iteration argument to derive a mean value inequality far away from the origin.  
Section~\ref{sec4} is devoted to the proofs of Theorems~\ref{thm1} and \ref{dim_Sch}. 
Section~\ref{sec5} constructs several examples.
Section~\ref{sec6} proves Theorem~\ref{polyrepre}, and 
Section~\ref{sec7} proves Theorem~\ref{dim_est_anc}.

\section{Preliminaries}\label{sec2}
In this section, we introduce some basic knowledge that is needed in our proof.
\subsection{Differential forms on Riemannian manifolds}
Let $\M$ be a complete Riemannian manifold, $\{e_i\}$ be its orthonormal frame and $\{\omega^i\}$ be its dual frame. 
We denote $R_{ijkl}=\langle Rm(e_i,e_j)e_l,e_k\rangle$ the components on Riemann curvature tensor. 
The Ricci tensor is given by $R_{ij}=\sum\limits_{k=1}^nR_{ikjk}$. 

Let $\Lambda^p(\M)$ be the linear space of $p$-forms on $\M$ , and $d:\Lambda^p(\M)\rightarrow\Lambda^{p+1}(\M)$ be the exterior differential operator. 
Recall that the Hodge star operator $\ast :\Lambda^p(\M)\rightarrow\Lambda^{n-p}(\M)$ is defined by $$\ast(\omega^{i_p}\wedge\cdots\wedge\omega^{i_1})=sgn(\sigma(I,I^C))\omega^{i_n}\wedge\cdots\wedge\omega^{i_{p+1}},$$
where $\sigma(I,I^C)$ denotes a permutation by sending $$(i_p,\cdots,i_1,i_n,\cdots,i_{p+1})\mapsto(1,2,\cdots,n).$$
The codifferential operator $\delta:\Lambda^p(\M)\rightarrow\Lambda^{p-1}(\M)$ is defined by $\delta=(-1)^{np+n+1}\ast d\ast$. 
It is well known that  when $\M$ is compact, $\delta$ is the adjoint operator of $d$ under the inner product on $\Lambda^p(\M)$, 
i.e., for any $\omega\in\Lambda^{p-1}(\M)$ and $\eta\in\Lambda^p(\M)$, 
we have $$\int\limits_\M\langle d\omega,\eta\rangle dx=\int\limits_\M\langle\omega,\delta\eta\rangle dx,$$ 
where $dx$ is the volume form on $\M$. We define the Hodge Laplacian by $\Delta_d\omega=-(d\delta+\delta d)\omega$.

Let $\omega\in\Lambda^p(\M)$ be an exterior $p$-form on $\M$. Then under orthogonal basis, we write $\omega=a_{i_1\cdots i_p}\omega^{i_p}\wedge\cdots\wedge\omega^{i_1}$. 
If we write $E(\omega)=R_{k_\beta i_\beta j_\alpha i_\alpha}a_{i_1\cdots k_\beta\cdots i_\alpha\cdots i_p}\omega^{i_p}\wedge\cdots\wedge\omega^{j_\alpha}\wedge\cdots\wedge\omega^{i_1}$, then we have the following formula, known as the Bochner-Weitzenb\"ock formula (see \cite[Lemma 3.4]{li2012geometric}, for instance):
$$\Delta|\omega|^2=2\langle\Delta_d\omega,\omega\rangle+2|\nabla\omega|^2+2\langle E(\omega),\omega\rangle.$$ We remark here that when $\omega=a_i\omega^i\in\Lambda^1(\M)$, then $E(\omega)=R_{ij}a_i\omega^j$, and hence $\langle E(\omega),\omega\rangle=\text{Ric}(\omega,\omega)$.

Now we will introduce the concept of curvature operator.
\begin{definition}
    The curvature operator of a Riemannian manifold is a linear map $\mathcal{R}:\Lambda^2(\M)\rightarrow\Lambda^2(\M)$ given by $$\mathcal{R}(\omega^i\wedge\omega^j)=R_{jikl}\omega^l\wedge\omega^k.$$
\end{definition}
Note that the curvature operator is a symmetric operator on linear space $\Lambda^2(\M)$, thus its eigenvalues are real. We emphasize that assumptions on curvature operator are usually strong assumptions on manifold. For example, if the curvature operator has a lower (resp. upper) bound, then by definition, the sectional curvature has the same lower (resp. upper) bound.
\subsection{A volume growth lemma}

In this subsection we will study the volume growth of the manifold. The following lemma shows that a complete manifold satisfying
$$\Ric \geq -\frac{\varepsilon}{1+r^2(x)}$$
has polynomial volume growth. Although this result is likely well known, we include a proof for the reader's convenience and for the sake of completeness.

\begin{lemma}\label{volume_growth}
    Suppose $\M$ is a complete manifold with $\Ric\geq-\frac{\varepsilon}{1+r^2(x)}$, where $r(x)=d(x,p)$ is the distance function to a fixed point $p\in\M$. Then we have the following geometric property:
    \begin{itemize}
        \item [(i)] Volume doubling property: for any $R>0$, we have \begin{equation}
            \Vol(B_p(2R))\leq 2^{n+\varepsilon}\Vol(B_p(R));
        \end{equation}
        \item [(ii)] Polynomial volume growth: for any $R>1$, we have \begin{equation}
            \Vol(B_p(R))\leq 2^{n+\varepsilon}\Vol(B_p(1))R^{n+\varepsilon}.
        \end{equation}
    \end{itemize}
\end{lemma}
\begin{proof}
For each $\theta\in S_p(\M)$, we define $c(\theta)$ the cut radius in the direction $\theta$.    
    
(i) Along the minimizing radial geodesic $\gamma_\theta$, for $0<r<c(\theta)$, we have \begin{equation}
        \begin{aligned}
            \Delta r&\leq\frac{n-1}{r}-\frac{1}{r^2}\int_0^rs^2\Ric(\pa_s,\pa_s)ds\\
            &\leq \frac{n-1}{r}+\frac{\varepsilon}{r^2}\int_0^r\frac{s^2}{1+s^2}ds\\
            &=\frac{n-1+\varepsilon}{r}-\frac{\varepsilon\arctan r}{r^2}.
        \end{aligned}
    \end{equation}
    Notice that $\Delta r=\frac{d}{dr}\log\sqrt{\det g}$, thus integration on $[r,R]\subseteq(0,c(\theta))$, we have \begin{equation}
        \log\sqrt{\det g}(R,\theta)-\log\sqrt{\det g}(r,\theta)\leq\int_r^{R}(\frac{n-1+\varepsilon}{t}-\frac{\varepsilon\arctan t}{t^2})dt\leq (n-1+\varepsilon)\log\big(\frac{R}{r}\big),
    \end{equation}
    which is \begin{equation}\label{eq2.5}
        \frac{\sqrt{\det g}(R,\theta)}{\sqrt{\det g}(r,\theta)}\leq \big(\frac{R}{r}\big)^{n-1+\varepsilon}.
    \end{equation}
    Now we define $$J(r,\theta):=\bigg\{\begin{aligned}
        &\sqrt{\det g}(r,\theta),\quad 0<r<c(\theta)\\
        &0,\quad r\geq c(\theta)
    \end{aligned}.$$ Then by (\ref{eq2.5}), we have $J(t,\theta)\leq\big(\frac{t}{s}\big)^{n-1+\varepsilon}J(s,\theta)$. 
    Since $$\Vol(B_p(R))=\int_{S_p(\M)}\int_0^RJ(r,\theta)drd\theta,$$ 
    therefore, we have \begin{equation}
       \begin{aligned}
            \frac{\Vol(B_p(2R))}{\Vol(B_p(R))}&\leq1+\frac{\int_{\mathbb{S}^{n-1}}\int_R^{2R}J(r,\theta)drd\theta}{\int_{\mathbb{S}^{n-1}}\int_0^RJ(r,\theta)drd\theta}\\
            &\leq 1+\frac{\int_{\mathbb{S}^{n-1}}\int_R^{2R}\big(\frac{r}{R}\big)^{n-1+\varepsilon}J(R,\theta)drd\theta}{\int_{\mathbb{S}^{n-1}}\int_0^{R}\big(\frac{r}{R}\big)^{n-1+\varepsilon}J(R,\theta)drd\theta}\\
            &=2^{n+\varepsilon},
       \end{aligned}
    \end{equation}
    which shows the volume doubling property holds.

    (ii) For $R>1$, there exists a unique integer $k$ such that $2^{k-1}<R\leq 2^k$. Thus by applying the volume doubling property $k$-times, we have \begin{equation}
        \Vol(B_p(R))\leq \Vol(B_p(2^k))\leq 2^{n+\varepsilon}\Vol(B_p(2^{k-1}))\leq\cdots\leq 2^{k(n+\varepsilon)}\Vol(B_p(1)).
    \end{equation}
    Since $k<\log_2R+1$, we obtain \begin{equation}
        \Vol(B_p(R))\leq 2^{(\log_2R+1)(n+\varepsilon)}\Vol(B_p(1))=2^{n+\varepsilon}\Vol(B_p(1))R^{n+\varepsilon},
    \end{equation}which shows that $\M$ has polynomial volume growth.
\end{proof}

\section{Mean value inequality far from origin}\label{sec3}
In this section, we establish a local mean value inequality, which plays an important role in our proof. We remark that this inequality holds for balls centered sufficiently far away from the origin.
\begin{theorem}\label{mvi}
    Let $\M$ be a complete manifold as in Theorem \ref{thm1}. Suppose $u\in C^\infty(\M)$ is a nonnegative smooth function satisfying $$\Delta u+\frac{2\varepsilon}{1+r^2}u\geq0,$$  then for all $R>0,0<\theta<\tau<\frac{1}{4},k>0$ and for all $x\in B_p(4R)\backslash B_p(R)$, there exists a constant $C>0$ such that the following mean value inequality holds:\begin{equation}
        \sup\limits_{B_x(\theta R)}|u|\leq C\bigg(\fint_{B_x(\tau R)}u^kdx\bigg)^{\frac{1}{k}}.
    \end{equation}
\end{theorem}
\begin{proof}
    First, we will prove the case when $k\geq 2$. Let $\phi$ be a cut-off function, which will be chosen later. For $a\geq1$, take $\phi^2u^{2a-1}$ as a test function, we obtain \begin{equation}
        \begin{aligned}
            0&\leq\int_\M\phi^2u^{2a-1}(\Delta u+\frac{2\varepsilon}{1+r^2}u)dx\\
            &=-\int_\M\langle 2\phi u^{2a-1}\nabla \phi+(2a-1)u^{2a-2}\phi^2\nabla u,\nabla u\rangle dx+\int_\M\frac{2\varepsilon}{1+r^2}\phi^2u^{2a}dx.
        \end{aligned}
    \end{equation}
    Choose $\phi$ such that $$\phi(y)=\bigg\{\begin{aligned}
        &1,\quad y\in B_x(\theta R)\\
        &0,\quad y\in B_x^C(\tau R)
    \end{aligned},\quad |\nabla \phi|\leq \frac{2}{(\tau-\theta)R}.$$
   By applying the triangle inequality, we obtain that $r(y)\geq (1-\tau)R$, thus by Cauchy-Schwarz inequality, we obtain  \begin{equation}
    \begin{aligned}
          (2a-1-\frac{1}{\delta_0})&\int_\M\phi^2u^{2a-2}|\nabla u|^2dx-\delta_0\int_\M u^{2a}|\nabla\phi|^2dx\\&\leq \int_\M\langle 2\phi u^{2a-1}\nabla \phi+(2a-1)u^{2a-2}\phi^2\nabla u,\nabla u\rangle dx\\
          &\leq\int_\M\frac{2\varepsilon}{1+r^2}\phi^2u^{2a}dx\leq\frac{2\varepsilon}{(1-\tau)^2R^2}\int_{B_x(\tau R)}u^{2a}(y)dy\leq\frac{32\varepsilon}{9R^2}\int_{B_x(\tau R)}u^{2a}(y)dy,
    \end{aligned}
    \end{equation} where $\delta_0 $ is a constant to be determined. Thus we obtain 
    \begin{equation}
        (2a-1-\frac{1}{\delta_0})\int_\M \phi^2u^{2a-2}(y)|\nabla u|^2(y)dy\leq \bigg(\frac{32\varepsilon}{9R^2}+\frac{4\delta_0}{(\tau-\theta)^2R^2}\bigg)\int_{B_x(\tau R)}u^{2a}(y)dy.
    \end{equation}
    Now let $w=u^a$, then this inequality becomes \begin{equation}
        \frac{2a-1-\frac{1}{\delta_0}}{a^2}\int_\M\phi^2|\nabla w|^2dy\leq\bigg(\frac{4\delta_0}{(\tau-\theta)^2}+\frac{32\varepsilon}{9}\bigg)\frac{1}{R^2}\int_{B_x(\tau R)}w^2dy
    \end{equation}
    By Cauchy-Schwarz inequality, we have \begin{equation}
        \begin{aligned}
            \frac{2a-1-\frac{1}{\delta_0}}{a^2}&\int_\M |\nabla(\phi w)|^2dx\leq\frac{2(2a-1-\frac{1}{\delta_0})}{a^2}\int_\M w^2|\nabla\phi|^2dx+\frac{2(2a-1-\frac{1}{\delta_0})}{a^2}\int_\M \phi^2|\nabla w|^2dx\\
            &\leq \bigg[\frac{8\delta_0}{(\tau-\theta)^2}+\frac{64\varepsilon}{9}+\frac{8(2a-1-\frac{1}{\delta_0})}{(\tau-\theta)^2a^2}\bigg]\frac{1}{R^2}\int_{B_x(\tau R)}w^2(y)dy.
        \end{aligned}
    \end{equation}
    By $L^2$-Sobolev inequality, we have \begin{equation}
        \frac{2a-1-\frac{1}{\delta_0}}{a^2}\int_{B_x(\tau R)}|\nabla(\phi w)|^2dy\geq \frac{2a-1-\frac{1}{\delta_0}}{a^2}\Vol(B_x(\tau R))\frac{1}{C_S\tau^2R^2}\bigg(\fint_{B_x(\tau R)}(\phi w)^{\frac{2n}{n-2}}dy\bigg)^{\frac{n-2}{n}}.
            \end{equation}
     Thus, we obtain \begin{equation}
    \begin{aligned}
    	\frac{2a-1-\frac{1}{\delta_0}}{a^2C_S\tau^2}&\bigg(\fint_{B_x(\theta R)}w^{\frac{2n}{n-2}}dy\bigg)^{\frac{n-2}{n}}\leq \bigg[\frac{64\varepsilon}{9}+\frac{8\delta_0}{(\tau-\theta)^2}+\frac{8(2a-1-\frac{1}{\delta_0})}{a^2(\tau-\theta)^2}\bigg]\\
    	&\times\bigg(\frac{\Vol(B_x(\tau R))}{\Vol(B_x(\theta R))}\bigg)^{\frac{n-2}{n}}\fint_{B_x(\tau R)}w^2dy
    \end{aligned}
            \end{equation}
    Pick $\delta_0=\frac{2}{a}$ and notice that $2a-1\geq a$ when $a\geq 1$, we have \begin{equation}
    	\bigg(\fint_{B_x(\theta R)}u^{a\frac{2n}{n-2}}\bigg)^{\frac{n-2}{2an}}\leq\bigg[2aC_S\tau^2\bigg(\frac{64\varepsilon}{9}+\frac{28}{(\tau-\theta)^2}\bigg)\bigg(\frac{\Vol(B_x(\tau R))}{\Vol(B_x(\theta R))}\bigg)^{\frac{n-2}{n}}\fint_{B_x(\tau R)}u^{2a}\bigg]^{\frac{1}{2a}}
    \end{equation}
    
    Now we use an iteration argument. Pick $$a_0=\frac{k}{2},\quad a_i=a_0\big(\frac{n}{n-2}\big)^i,\quad\tau_0=\tau,\quad\tau_i=\tau-\sum\limits_{j=1}^i2^{-j}(\tau-\theta).$$ Then we obtain \begin{equation}
        \begin{aligned}
            \bigg(\fint_{B_x(\tau_{i+1}R)}u^{2a_{i+1}}\bigg)^{\frac{1}{2a_{i+1}}}\leq\bigg[2a_iC_S\tau_i^2\bigg(\frac{64\varepsilon}{9}+\frac{28\cdot 4^{i+1}}{a_i(\tau-\theta)^2}\bigg)\bigg(\frac{\Vol(B_x(\tau_i R))}{\Vol(B_x(\tau_{i+1} R))}\bigg)^{\frac{n-2}{n}}\fint_{B_x(\tau_i R)}u^{2a_i}\bigg]^{\frac{1}{2a_i}}
        \end{aligned}
    \end{equation}
    By induction, we have \begin{equation}\label{3.10}
        \begin{aligned}
            \bigg(\fint_{B_x(\tau_{i+1}R)}u^{2a_{i+1}}\bigg)^\frac{1}{2a_{i+1}}&\leq\bigg(\fint_{B_x(\tau R)}u^k\bigg)^\frac{1}{k}\prod\limits_{j=0}^i\bigg[2a_iC_S\tau_i^2\bigg(\frac{64\varepsilon}{9}+\frac{28\cdot 4^{i+1}}{a_i(\tau-\theta)^2}\bigg)\bigg]^{\frac{1}{2a_i}}\\
            &\times\prod\limits_{j=0}^i\bigg(\frac{\Vol(B_x(\tau_i R))}{\Vol(B_x(\tau_{i+1} R))}\bigg)^{\frac{1}{2a_{i+1}}}.
        \end{aligned}
    \end{equation}
    Notice that \begin{equation}
        \begin{aligned}
            &\sum\limits_{j=0}^\infty\frac{1}{a_i}=\frac{n}{k}<\infty;\\
            &\prod\limits_{i=0}^\infty (2a_i)^{\frac{1}{2a_i}}=\exp\bigg(\sum\limits_{i=0}^\infty \frac{\log(2a_i)}{2a_i}\bigg)\leq C(n,k);\\
            &\prod\limits_{i=0}^\infty (\tau_i)^{\frac{1}{a_i}}=\exp\bigg(\sum\limits_{i=0}^\infty \frac{\log(\tau_i)}{a_i}\bigg)\leq \exp\bigg(\sum\limits_{i=0}^\infty \frac{\log(\tau)}{a_i}\bigg)\leq C(n,k);\\
            &\prod\limits_{i=0}^\infty\bigg(\frac{64\varepsilon}{9}+\frac{28\cdot4^{i+1}}{a_i(\tau-\theta)^2}\bigg)^{\frac{1}{2a_i}}=\frac{1}{(\tau-\theta)^{\frac{n}{k}}}\exp\bigg(\sum\limits_{i=0}^\infty\frac{1}{2a_i}\log\big(\frac{64\varepsilon(\tau-\theta)^2}{9}+\frac{28\cdot4^{i+1}}{a_i}\big)\bigg)\leq \frac{C(\varepsilon,n,k)}{(\tau-\theta)^{\frac{n}{k}}}.
        \end{aligned}
    \end{equation}
    Besides, notice that for $x\in B_p(4R)\backslash B_p(R)$ and $y\in B_x(\tau_i R)$, we have $$r(y)\geq d(x,p)-d(x,y)\geq (1-\tau_i)R\geq\frac{3R}{4},$$ which implies $\Ric(y)\geq-\frac{16\varepsilon}{9R^2}g$. 
    Thus by Bishop-Gromov volume comparison theorem, we obtain \begin{equation}
    	\begin{aligned}
    		\prod\limits_{i=0}^\infty&\bigg(\frac{\Vol(B_x(\tau_i R))}{\Vol(B_x(\tau_{i+1}R))}\bigg)^{\frac{1}{2a_{i+1}}}=\exp\bigg[\sum\limits_{i=0}^\infty\frac{1}{2a_{i+1}}\log\bigg(\frac{\Vol(B_x(\tau_i R))}{\Vol(B_x(\tau_{i+1}R))}\bigg)\bigg]\\
    		&=\exp\bigg\{\sum\limits_{i=0}^\infty\frac{1}{k}\bigg(\frac{n-2}{n}\bigg)^{i+1}\bigg[\log\bigg(C(n)\big(\frac{\tau_i}{\tau_{i+1}}\big)^n\bigg)+(n-1)\sqrt{\frac{16\varepsilon}{9R^2}}(\tau_i-\tau_{i+1})R\bigg]\bigg\}\\
    		&\leq \exp\bigg\{\sum\limits_{i=0}^\infty\frac{1}{k}\bigg(\frac{n-2}{n}\bigg)^{i+1}\bigg[\log\bigg(C(n)\big(\frac{\tau}{\theta}\big)^n\bigg)+(n-1)\sqrt{\frac{16\varepsilon}{9R^2}}(\tau_i-\tau_{i+1})R\bigg]\bigg\}\\
    		&\leq C(n,\varepsilon,\tau,\theta).
    	\end{aligned}
    \end{equation}
    Thus, by taking $j\rightarrow\infty$ in (\ref{3.10}), we obtain \begin{equation}\label{3.12}
        \sup\limits_{B_x(\theta R)}|u|\leq \frac{C(k,C_S,n,\theta,\tau)}{(\tau-\theta)^{\frac{n}{k}}}\bigg(\fint_{B_x(\tau R)}u^k\bigg)^{\frac{1}{k}}.
    \end{equation}

    Finally we will prove the case when $k<2$. Take $k=2$ in (\ref{3.12}) and apply Young's inequality, we have \begin{equation}
        \begin{aligned}
            \sup\limits_{B_x(\theta R)}|u| &\leq \frac{C}{(\tau-\theta)^n}\bigg(\fint_{B_x(\tau R)}u^2\bigg)^{\frac{1}{2}}\\
            &\leq \frac{C}{(\tau-\theta)^n}\bigg(\fint_{B_x(\tau R)}u^k\bigg)^{\frac{1}{2}}(\sup\limits_{B_x(\tau R)}|u|)^{1-\frac{k}{2}}\\
            &\leq \frac{1}{2}\sup\limits_{B_x(\tau R)}|u|+\frac{C}{(\tau-\theta)^{\frac{n}{k}}}\bigg(\fint_{B_x(\tau R)}u^k\bigg)^{\frac{1}{k}}.
        \end{aligned}
    \end{equation}
    The result now follows by applying \cite[Lemma 4.3]{han2011elliptic}.
\end{proof}

\section{Counting dimension}\label{sec4}
In this section, we will prove Theorem \ref{thm1} and Theorem \ref{dim_Sch}. 
For convenience, we take the case of differential forms to state our idea. 
Our idea is to give a uniform upper bound of the dimension of the subspace $K\subseteq \mathcal{H}_l^1(\M)$, where $\mathcal{H}_l^1(\M)$ denotes the linear space of harmonic $1$-forms of polynomial growth of order $l$ on $\M$. 
On the other hand, the Bochner-Weitzenb\"ock formula tells us that the curvature term plays a role as a potential term when considering $\Delta|\omega|^2$. 
Thus, for the case of Schr\"odinger equation, our argument is quite similar, just by replacing the norm of differential forms by $|u|$, and replacing the curvature term by the potential term in Schr\"odinger equation. 
\begin{lemma}
    Given $R>R_0>0$ and a subspace $K\subseteq \mathcal{H}_l^1(\M)$, the bilinear form $$A_{R,R_0}(\omega,\nu):=\int_{B_p(R)\backslash B_p(R_0)}\langle \omega,\nu\rangle dx$$ is an inner product on $K$.
\end{lemma} 
\begin{proof}
    What we need to show is that $A_{R,R_0}(\omega,\omega)=0$ implies $\omega\equiv0$ on $\M$. 
    This is exactly the unique continuation property. 
    This follows from the unique continuation result in \cite[Theorem 3]{wolff1992property}, together with the discussion in the two paragraphs immediately following the theorem.\end{proof}

\begin{lemma}\label{lower_bound_esti_form}
    Let $\M$ be a complete manifold with volume condition $\Vol(B_p(R))\leq CR^{d_V}$ for some $p\in\M$ and $R$ large. 
    Let $K$ be any $k$-dimensional subspace of $\mathcal{H}_l^1(\M)$. 
    For any $\beta>1,\delta_1>0, R_0>0$, there exists $R>R_0$, which may depend on $k$ and $R_0$, 
    such that if $\{\omega_i\}_{i=1}^k$ is an orthonormal basis of $K$ with respect to the inner product $A_{\beta R,R_0}$, 
    then we have \begin{equation}
        \sum\limits_{i=1}^k\int_{B_p(R)\backslash B_p(R_0)}|\omega_i|^2dx\geq k\beta^{-(2l+d_V+\delta_1)}.
    \end{equation}
\end{lemma}
\begin{proof}
    For a positive integer $J$, we denote $D_J:=B_p(\beta^J)\backslash B_p(R_0)$ and the $k\times k$ matrix $$M_J:=\bigg(\int_{D_J}\langle \omega_i,\omega_j\rangle dx\bigg).$$
     By the growth condition on $\omega_i$ and volume of geodesic balls, we have \begin{equation}\label{4.2}
         \det M_J\leq C\beta^{kJ(2l+d_V)}k!.
     \end{equation}
     We claim that for any positive integer $k$, there exists an integer $J=J(k,\beta,\delta_1,l,d_V)$ such that \begin{equation}\label{4.3}
         \det M_{J+1}\leq \beta^{k(2l+d_V+\delta_1)}\det M_J.
     \end{equation}
     If the claim is not true, then there exists an integer $k$ such that for all $J$, we have $$\det M_{J+1}>\beta^{k(2l+d_V+\delta_1)}\det M_J.$$
     Thus by induction on $J$, we obtain that for all $J$, $$\det M_J>\beta^{k(J-1)(2l+d_V+\delta_1)}\det M_1,$$ 
     which will lead to a contradiction with (\ref{4.2}) when $J$ is large. 
     Thus our claim holds. 
     By normalization, we assume that $\{\omega_i\}_{i=1}^k$ is an orthonormal basis under the inner product $A_{\beta^{J+1},R_0}$ so that $\det M_{J+1}=1$. 
     Applying the arithmetic-geometric inequality in (\ref{4.3}), we obtain $$k\leq \beta^{2l+d_V+\delta_1}\text{Tr}M_J,$$ where $\text{Tr} M_J$ is the trace of $M_J$. 
     And our result follows from taking $R=\beta^J$.
\end{proof}

Before we prove our main Theorem \ref{thm1}, we need the following lemma, which implies the volume of annuli $B_p(2R)\backslash B_p(R)$ can be bounded by volume of geodesic balls far away from origin. 
\begin{lemma}\label{vol_comparison_far}
	Suppose $\M$ is a complete manifold satisfying (a) and (c) in Theorem \ref{thm1}. Then there exists a constant $C_0$ depending only on $n, C_S$ and $\varepsilon$ such that for all $x\in B_p(2R)\backslash B_p(R)$, we have\begin{equation}
		\Vol(B_p(2R)\backslash B_p(R))\leq C_0\sigma^{-n} \Vol(B_x(\sigma R)),
	\end{equation} where $0<\sigma\leq\frac{1}{2}$ is a fixed constant. 
\end{lemma}
\begin{proof}
	For $x\in B_p(2R)\backslash B_p(R)$, we pick a smooth cut-off function $\eta$ such that \begin{equation}
		0\leq\eta\leq1,\quad \eta=\Bigg\{\begin{aligned}
			&1,\quad B_x(\frac{\sigma R}{2}),\\
			&0,\quad \M\backslash B_x(\sigma R),
		\end{aligned}\quad |\nabla\eta|\leq\frac{4}{\sigma R}.
	\end{equation}
	Thus by our assumption (a), we have \begin{equation}
		\bigg(\fint_{B_p(3R)}|\eta|^{\frac{2n}{n-2}}dx\bigg)^{\frac{n-2}{n}}\leq C_S(3R)^2\fint_{B_p(3R)}|\nabla\eta|^2dx.
	\end{equation}
	By our choice of $\eta$, we obtain that \begin{equation}
		\bigg(\fint_{B_p(3R)}|\eta|^{\frac{2n}{n-2}}dx\bigg)^{\frac{n-2}{n}}\geq \bigg(\frac{\Vol(B_x(\frac{\sigma R}{2}))}{\Vol(B_p(3R))}\bigg)^{\frac{n-2}{n}}.
	\end{equation}
	and \begin{equation}
		C_S(3R)^2\fint_{B_p(3R)}|\nabla\eta|^2dx\leq 9C_SR^2\frac{\Vol(B_x(\sigma R))}{\Vol(B_p(3R))}\bigg(\frac{4}{\sigma R}\bigg)^2.
	\end{equation}
	Thus we obtain that \begin{equation}\label{eq4.9}
		\Vol(B_p(3R))\leq(\frac{144C_S}{\sigma^2})^{\frac{n}{2}}\Vol(B_x(\sigma R))\bigg(\frac{\Vol(B_x(\sigma R))}{\Vol(B_x(\frac{\sigma R}{2}))}\bigg)^{\frac{n-2}{2}}.
	\end{equation}
	Notice that for $y\in B_x(\sigma R)$, by triangle inequality, we obtain \begin{equation}
		r(y)\geq d(p,x)-d(x,y)\geq (1-\sigma)R,
	\end{equation}
	which implies $$\Ric(y)\geq -\frac{\varepsilon}{1+r^2(y)}\geq-\frac{\varepsilon}{(1-\sigma)^2R^2}\geq-\frac{4\varepsilon}{R^2}.$$
	Thus by Bishop-Gromov volume comparison theorem, we have \begin{equation}\label{eq4.11}
		\frac{\Vol(B_x(\sigma R))}{\Vol(B_x(\frac{\sigma R}{2}))}\leq \frac{\bar{\Vol}(\bar{B}(\sigma R))}{\bar{\Vol}(\bar{B}(\frac{\sigma R}{2}))}=c(n,\varepsilon).
	\end{equation}
	Since $$B_p(2R)\backslash B_p(R)\subseteq B_p(3R),$$ therefore, this lemma follows by combining (\ref{eq4.9}) and (\ref{eq4.11}).
\end{proof}

Now we begin to prove Theorem \ref{thm1}.
\begin{proof}
    Let $K$ be a $k$-dimensional subspace of $\mathcal{H}_l^1(\M)$. 
    If $k\leq n$, then there is nothing to prove. So we assume that $k>n$.
    Let $0<\lambda<\frac{1}{2}$ be a fixed constant to be determined. 
    Since $\M$ has polynomial volume growth of order $n+\varepsilon$, thus by taking $\beta=1+2\lambda$ in Lemma \ref{lower_bound_esti_form}, we have \begin{equation}
        k\leq (1+2\lambda)^{2l+n+\varepsilon+\delta_1}\sum\limits_{i=1}^k\int_{B_p(R)\backslash B_p(R_0)}|\omega_i|^2dx,
    \end{equation} where $\{\omega_i\}_{i=1}^k$ is an orthonormal basis with respect to the inner product $A_{(1+2\lambda)R,R_0}$. By Lemma \ref{lem_appendix}, we have 
    \begin{equation}
        k\leq C(1+2\lambda)^{2l+n+\varepsilon+\delta_1}\frac{(1+\frac{1}{2}\lambda)^2}{\lambda^2}\sum\limits_{i=1}^k\int_{B_p((1+\lambda)R)\backslash B_p(R)}|\omega_i|^2dx.
    \end{equation}
    For $x\in B_p((1+\lambda)R)\backslash B_p(R)$, we let $K_x:=\{\omega\in K\big|\omega(x)=0\}$, then by \cite[Lemma 11]{li1980sobolev}, $K_x$ has codimension at most $n$ (for $q$-forms, $K_x$ has codimension at most $(^n_q)$). Then there exists an orthogonal transformation mapping $\{\omega_i\}_{i=1}^k$ into $\{\nu_i\}_{i=1}^k$, where $\{\nu_i\}_{i=n+1}^k\subseteq K_x$. Thus by our local mean value inequality, we have \begin{equation}
        \sum\limits_{i=1}^k|\omega_i|^2(x)=\sum\limits_{i=1}^n|\nu_i|^2(x)\leq C\sum\limits_{i=1}^k\fint_{B_x((1+2\lambda)R-d(x,p))}|\nu_i|^2dx.
    \end{equation}
    Observe that when $R$ is large and $\lambda$ is small, we always have $(1-2\lambda)R>R_0.$ Thus by our assumption on $\{\omega_i\}_{i=1}^k,$ we obtain that \begin{equation}
        \begin{aligned}
            \fint_{B_x((1+2\lambda)R-d(x,p))}|\nu_i|^2(y)dy&\leq \frac{C}{\Vol(B_x((1+2\lambda)R-d(x,p))}\sup\limits_{\omega\in \langle A,U\rangle}\int_{B_x((1+2\lambda)R-d(x,p))}|\omega|^2dy\\
            &\leq  \frac{C}{\Vol(B_x((1+2\lambda)R-d(x,p))},
        \end{aligned}
    \end{equation}
    where $\langle A,U\rangle=\{\sum\limits_{i=1}^ka_i\omega_i|\sum\limits_{i=1}^ka_i^2=1\}.$

    Integration on $B_p((1+\lambda)R)\backslash B_p(R)$, we obtain that \begin{equation}\label{eq5.8}
        k\leq (1+2\lambda)^{2l+n+\varepsilon+\delta_1}Cn\frac{(1+\frac{1}{2}\lambda)^2}{\lambda^2}\int_{B_p((1+\lambda)R)\backslash B_p(R)}\frac{dx}{\Vol(B_x((1+2\lambda)R-d(x,p)))}.
    \end{equation}
    Notice that $(1+2\lambda)R-d(x,p)\geq \lambda R$, thus, we have \begin{equation}
    	\int_{B_p((1+\lambda)R)\backslash B_p(R)}\frac{dx}{\Vol(B_x((1+2\lambda)R-d(x,p)))}\leq \int_{B_p((1+\lambda)R)\backslash B_p(R)}\frac{dx}{\Vol(B_x(\lambda R))}.
    \end{equation}
    Pick $\lambda=\sigma=\frac{1}{16}$, and apply our Lemma \ref{vol_comparison_far}, we have \begin{equation}
    	\int_{B_p((1+\lambda)R)\backslash B_p(R)}\frac{dx}{\Vol(B_x((1+2\lambda)R-d(x,p)))}\leq\frac{\Vol(B_p(2R)\backslash B_p(R))}{\inf\limits_{x\in B_p(2R)\backslash B_p(R)}\Vol(B_x(\sigma R))}\leq C_1,
    \end{equation}
    which implies $k\leq C$ for some constant $C$ independent of $K$. Thus our result follows.
\end{proof}

In the Bochner formula, the curvature term plays the role of a potential. 
Therefore, the same argument proves Theorem \ref{dim_Sch}, and we omit the details.
\section{Examples}\label{sec5}
\subsection{Example for our assumption}\label{subsec5.1}
In this subsection, we will give an example to show that there indeed exist some manifolds satisfying our assumptions in Theorem \ref{thm1}. 
Consider the warped product $$(\M,g)=(\mathbb{R}^n,g_\delta=dr^2+r^2(1+\delta\chi(r)))^2g_{\mathbb{S}^{n-1}}),$$ 
where $\chi$ is a fixed smooth function satisfying $$\chi=\bigg\{\begin{aligned}
&0,\quad r\leq 1\\
&1,\quad r\geq 2
\end{aligned}$$ and $\delta$ is a small constant. Then $(\mathbb{R}^n,g_\delta)$ is a complete, nonflat manifold with one end.
Moreover, it's easy to verify that $g_\delta$ is uniformly equivalent to the standard Euclidean metric on $\mathbb{R}^n$. 
Thus by Chen-Sung \cite{chen2006dimension}, the space of harmonic 1-forms with polynomial growth of some finite order has finite dimension, and dimension estimate is also established there.
On the other hand, by the uniform equivalence of the two metric, we know that the $L^2$-Sobolev inequality also holds on our manifold.

Now we will show that the following weak Hardy inequality holds for all $u\in C^\infty_0(\M)$:\begin{equation}
    \int_\M\frac{u^2}{1+r^2(x)}dx\leq\frac{4(1+\delta)^{n+1}}{(n-2)^2}\int_\M |\nabla u|^2dx.
\end{equation}
To see this, we denote $Du$ the gradient of $u$ on standard Euclidean space and $d\mu$ the volume form on Euclidean space.
Then by our construction, it's easy to see that \begin{equation}
    dx=(1+\delta\chi(r))^{n-1}d\mu,\quad |\nabla u|^2_{g_\delta}<|Du|^2_{Euc}<(1+\delta\chi(r))^2|\nabla u|^2_{g_\delta}.
\end{equation}
Thus we obtain \begin{equation}
    \begin{aligned}
        \int_\M\frac{u^2}{1+r^2(x)}dx&=\int_{\mathbb{R}^n}\frac{u^2}{1+r^2_g(x)}dx\\
        &\leq (1+\delta)^{n-1}\int_{\mathbb{R}^n}\frac{u^2}{1+r^2_g(x)}d\mu\\
        &\leq (1+\delta)^{n-1}\int_{\mathbb{R}^n}\frac{u^2}{1+|x|^2}d\mu\\
        &\leq \frac{4(1+\delta)^{n-1}}{(n-2)^2}\int_{\mathbb{R}^n}|Du|^2_{Euc}d\mu\\
        &\leq \frac{4(1+\delta)^{n+1}}{(n-2)^2}\int_{\mathbb{R}^n}|\nabla u|^2_gdx.
    \end{aligned}
\end{equation}

Finally we study the Ricci curvature of this manifold. By direct computation, we have \begin{equation}
    \begin{aligned}
        \Ric(\partial_r)&=-(n-1)\frac{2\delta\chi'(r)+r\delta\chi''(r)}{r(1+\delta\chi(r))}\partial_r\\
        \Ric(X)&=\bigg[(n-2)\frac{1-(1+\delta\chi(r)+r\delta\chi'(r))^2}{r^2(1+\delta\chi(r))^2}-\frac{2\delta\chi'(r)+r\delta\chi''(r)}{r(1+\delta\chi(r))}\bigg]X
    \end{aligned}
\end{equation}
When $0<r<1$, Ricci curvature vanishes. When $r\geq 2$, we see that $\chi(r)=1,\chi'(r)=\chi''(r)=0$. Thus in this case, we have \begin{equation}
    \Ric(\partial_r)=0, \Ric(X)=-(n-2)\frac{\delta(\delta+2)}{r^2(1+\delta)^2}X.
\end{equation}
Thus Ricci curvature has quadratic decay. When $1<r<2$, we denote $A_1=||\chi'||_{L^\infty},A_2=||\chi''||_{L^\infty}$.
Then we have \begin{equation}
    \begin{aligned}
        \Ric(\partial_r)&\geq-\frac{(n-1)(2\delta A_1+r\delta A_2)}{r}\geq-\frac{\delta C_1(n,\chi)}{1+r^2}\\
        \Ric(X)&\geq-(n-2)\frac{(2+2\delta+r\delta A_1)(2\delta+r\delta A_1)}{r^2}-\frac{2\delta A_1+r\delta A_2}{r}\geq-\frac{\delta C_2(n,\chi)}{1+r^2}.
    \end{aligned}
\end{equation}
Thus we obtain that there exists a constant $C$ depending only on $n$ and $\chi$ such that \begin{equation}
    \Ric\geq-\frac{C\delta}{1+r^2(x)}.
\end{equation}
When $\delta$ is sufficiently small, we see that $C\delta \frac{4(1+\delta)^{n+1}}{(n-2)^2}<1$ must hold, which implies there indeed exists a complete manifold satisfying our assumptions in Theorem \ref{thm1}.
\subsection{Counterexample for small curvature}\label{subsec5.2}
In this subsection, we present a counterexample to show that the small curvature assumption is necessary. 
Roughly speaking, our counterexample is constructed by gluing countably many handles smoothly on the Euclidean space, see the following figure \ref{Figure 1}.

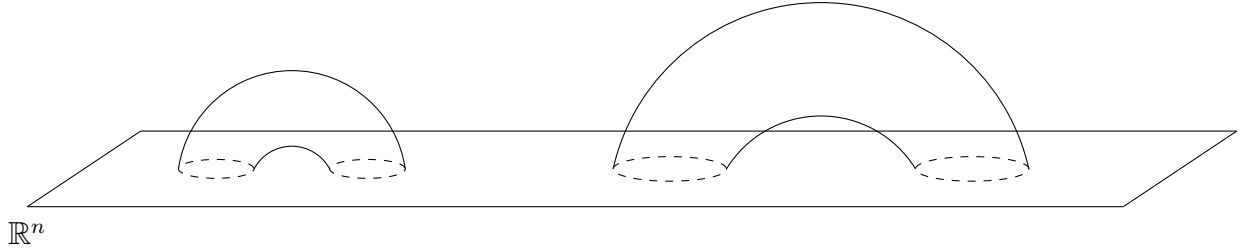
\begin{figure}[H]
    \begin{tikzpicture}
        \draw (0,0)--(14.5,0);
        \draw (1.5,1)--(16,1);
        \draw (0,0)--(1.5,1);
        \draw (14.5,0)--(16,1);
        \draw[dashed] (2.5,0.5) ellipse (0.5 and 0.125);
        \draw[dashed] (4.5,0.5) ellipse (0.5 and 0.125);
        \draw[use Hobby shortcut] (2,0.5)..(3.5,1.8)..(5,0.5);
        \draw[use Hobby shortcut] (3,0.5)..(3.5,0.8)..(4,0.5);
        \draw[dashed] (8.5,0.5) ellipse (0.75 and 0.1625);
        \draw[dashed] (12.5,0.5) ellipse (0.75 and 0.1625);
        \draw[use Hobby shortcut] (7.75,0.5)..(10.5,2.7)..(13.25,0.5);
        \draw[use Hobby shortcut] (9.25,0.5)..(10.5,1.2)..(11.75,0.5);
        \node at (0,-0.35) {$\mathbb{R}^n$};
    \end{tikzpicture}
    \caption{Counterexample for small curvature}
    \label{Figure 1}
\end{figure}

Precisely speaking, consider the $k$-th pair balls in Euclidean space $B_{\mathbb{R}^n}(a_ke_1,\frac{1}{2}a_k)$ and $B_{\mathbb{R}^n}(4a_ke_1,\frac{1}{2}a_k)$, where $a_k=4^{2k-1},e_1=(1,0,...,0)\in\mathbb{R}^n$.
We cut out these balls and smoothly attach a handle $H_k$ from the first ball of $k$-th pair to the second one, where these handles are self-similar to each other.
Then let $$Q_k:=\bigg(\{x\in\mathbb{R}^n|\frac{a_k}{3}\leq |x|\leq \frac{16a_k}{3}\}\backslash\big(B_{\mathbb{R}^n}(a_ke_1,\frac{1}{2}a_k)\cup B_{\mathbb{R}^n}(4a_ke_1,\frac{1}{2}a_k)\big)\bigg)\bigcup H_k.$$
Now we define the manifold $\M=K_0\cup\bigcup\limits_{k=1}^\infty Q_k$, where $K_0=\bar{B}_{\mathbb{R}^n}(o,\frac{a_1}{3})$.
By our self-similar assumption on such manifold, there exists a fixed block $(Q,h)$ such that $(Q_k,g)\cong (Q,a_k^2h)$ is isometric.
\begin{figure}[H]
    \begin{tikzpicture}
        \draw[dashed] (8,0) ellipse (2 and 0.5);
        \draw[dashed] (8,0) ellipse (8 and 2);
        \draw[dashed] (8,0) ellipse (0.125 and 0.0375);
        \draw[dashed] (8.5,0) ellipse (0.25 and 0.0625);
        \draw[dashed] (9.5,0) ellipse (0.25 and 0.0625);
        \draw[dashed] (11,0) ellipse (0.5 and 0.25);
        \draw[dashed] (15,0) ellipse (0.5 and 0.25);
        \draw[use Hobby shortcut] (10.5,0)..(13,2)..(15.5,0);
        \draw[use Hobby shortcut] (11.5,0)..(13,1)..(14.5,0);
        \draw[use Hobby shortcut] (8.25,0)..(9,0.65)..(9.75,0);
        \draw[use Hobby shortcut] (8.75,0)..(9,0.325)..(9.25,0);
        \node at (8,-1.5) {$(Q,h)$};
    \end{tikzpicture}
    \caption{Schematic figure of standard block $(Q,h)$}
    \label{Figure 2}
\end{figure}
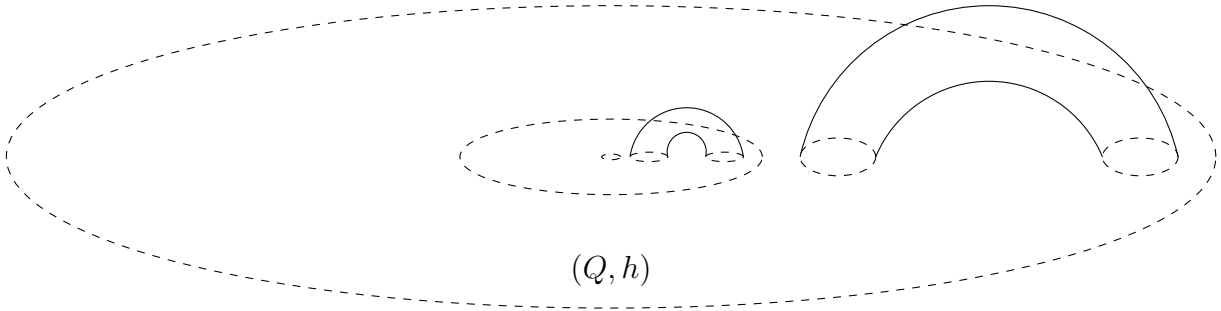
We denote $\partial_-Q=\partial B_{\mathbb{R}^n}(o,\frac{1}{3}),\partial_+Q=\partial B_{\mathbb{R}^n}(o,\frac{16}{3})$, and $\delta:=d_Q(\partial_-Q,\partial_+Q)>0$, $D=\operatorname{diam}_h(Q)<\infty$. And $\partial_-Q_k,\partial_+Q_k$ will be the corresponding part of $\partial Q_k$.

We now prove that this example is indeed a counterexample for our small curvature assumption. 
In order to verify assumption (a) and (b) in Theorem \ref{thm1}, we prove that such a manifold is an $\operatorname{SOB}(n)$ manifold in our weak sense.
If this is true, then by following the proof of \cite[Theorem 1.2]{hein2011weighted}, we immediately establish our assumption (a) and (b) both hold by choosing $\alpha=\frac{n}{n-2},\beta=n$ and $\alpha=1,\beta=n$, respectively.

For (i) in Definition \ref{weakSOB}, since $a_{k+1}=16a_k$ and the blocks $Q_k$ are mutually homothetic, we may choose $r_0>0$ such that $\partial B_o(r_0)$ lies in the interior of $Q_1$, away from its two boundary components. 
By self-similarity and the serial structure of the blocks, the corresponding distance levels $16^{k-1}r_0$ lie in the interiors of $Q_k$. 
Moreover, the region between two consecutive such levels is contained in the union of two adjacent self-similar blocks and is connected. 
Hence $A(o,16^{k-1}r_0,16^kr_0)$ is connected for every $k\ge1$.

Before we begin to prove (ii) and (iii) in Definition \ref{weakSOB}, we first prove that there exist constants $c,C_0$ such that for all $x\in Q_k$, we have $ca_k\leq r(x)\leq C_0a_k$. 
For $k\geq 2$, and for any path $\gamma$ from $x_0$ to some point $x\in Q_k$, $\gamma$ has to pass through $Q_{k-1}$. 
We pick the segment that $\gamma$ lies in $Q_{k-1}$. 
If $\gamma$ passes through $Q_{k-1}$ several times, we pick the segment that last leaves $\partial_-Q_{k-1}$ and first reaches $\partial_+Q_{k-1}$.
Then it's easy to see that the length of this segment is at least $\delta a_{k-1}$. 
Thus, we obtain $r(x)\geq \frac{\delta}{16}a_k$.
For the case $k=1$, it's easy to see that $r(x)\geq\frac{1}{3}a_1$. Thus we pick $c=\min\{\frac{1}{3},\frac{\delta}{16}\}$ to establish the lower bound.
And for upper bound, by our self-similar structure of $Q_k$, we have \begin{equation}
    r(x)\leq \frac{1}{3}a_1+D\sum\limits_{j=1}^ka_j\leq\bigg(\frac{1}{3}+\frac{16}{15}D\bigg)a_k,
\end{equation} which implies that our distance estimate holds.

Now we are going to derive the volume estimates. First of all, we derive the upper bound of geodesic balls. 
Let $V:=\Vol(Q), V_0:=\Vol(K_0)$ be volume of standard block and center ball, respectively.
Then by our self-similar assumption, $\Vol(Q_k)=Va_k^n$. Therefore, we have \begin{equation}
    \begin{aligned}
        \Vol(B_{o}(R))&\leq V_0+V\sum\limits_{a_k\leq\frac{R}{c}}a_k^n\\
        &\leq V_0+V\sum\limits_{a_k\leq\frac{R}{c}}4^{(2k-1)n}=V_0+4^nV\frac{16^{n(k_0+1)}-1}{16^n-1}\\
        &\leq V_0+4^nV\frac{16^n}{16^n-1}\bigg(\frac{4R}{c}\bigg)^n\leq CR^n,
    \end{aligned}
\end{equation}
where $k_0$ is the largest number such that $a_k\leq\frac{R}{c}$. This finishes our proof of (ii) in Definition \ref{weakSOB}.

For lower bound volume estimate, first we expand $Q_k$ into $U_k$ by the following way: \begin{equation}
    U_k=Q_k\cup\{\frac{5}{16}a_k\leq|x|\leq\frac{a_k}{3}\}\cup\{\frac{16}{3}a_k\leq|x|\leq6a_k\}.
\end{equation}
Then it's easy to verify that $U_k\cap\partial H_k=\emptyset$ and $(U_k,g)\cong(U,a_k^2h)$. 
Now for a nonempty subset $\Omega$, we denote $B_\Omega(y,t)$ be the set of those points $z$ such that $z$ and $y$ can be connected by a piecewise smooth curve $\gamma\subset\Omega$ with length $\leq t$.
Then for any $x\in Q$ and $t\in(0,1]$, we claim that $\Vol(B_U(x,t))\geq \nu t^n$. 
By compactness of $Q$, we can pick finitely many local charts $\phi_i:B_{\mathbb{R}^n}(0,2)\rightarrow U$ such that $Q\subseteq\bigcup\limits_i\phi_i(B_{\mathbb{R}^n}(0,1))$. 
Then by finiteness, we can pick a constant $L\geq 1$ such that $L^{-1}g_{\mathbb{R}^n}\leq \phi_i^*h\leq Lg_{\mathbb{R}^n}$ holds for any $i$.
For $y=\phi_i(z),|z|<1$ and $t\leq 1$, we have $B_{\mathbb{R}^n}(z,\frac{t}{\sqrt{L}})\subseteq B_{\mathbb{R}^n}(0,2)$, which implies $\phi_i(B_{\mathbb{R}^n}(z,\frac{t}{\sqrt{L}}))\subseteq B_U(y,t)$.
Therefore, since $g_{\mathbb{R}^n}$ and $\phi_i^*h$ is uniformly equivalent, we have \begin{equation}
    \Vol(B_x(t))\geq \Vol(B_U(x,t))\geq\Vol(\phi_i(B_{\mathbb{R}^n}(z,\frac{t}{\sqrt{L}})))\geq \omega_nL^{-n}t^n.
\end{equation}
For $k$-th block, for $x\in Q_k$, and $0\leq t\leq a_k$, we have \begin{equation}
    \Vol(B_x(t))\geq a_k^n\Vol(B_U(y,\frac{t}{a_k}))\geq a_k^n\omega_nL^{-n}(\frac{t}{a_k})^n=\omega_nL^{-n}t^n.
\end{equation}
Now we pick $\theta=\min\{\frac{1}{2},C_0^{-1}\}$, then $r(x)\leq C_0a_k$ implies $\theta r(x)\leq a_k$. 
Therefore, we obtain that $\Vol(B_x((1-\frac{1}{C})r(x)))\geq\omega_nL^{-n}\theta^nr(x)^n$ by choosing $C$ sufficiently large.

Finally we show that the Ricci curvature of our manifold has quadratic decay, which finishes our proof of (iii) in Definition \ref{weakSOB} . To see this, notice that the curvature vanishes outside those handles.
Thus it suffices to consider the case on those handles. Since $H$, the first handle, is compact, thus the Ricci curvature has a lower bound $-\kappa$ there.
By our self-similar assumption, on the $k$-th handle, it's easy to see that $\Ric\geq-\kappa a_k^{-2}$.
Notice that on points $x \in H_k$, we have $r(x)\leq C_0a_k$. Therefore, we have $\Ric(x)\geq -C_0^{2}\kappa r(x)^{-2}$, which finishes the proof of (iii) in Definition \ref{weakSOB}.

By our construction, such a manifold is complete and noncompact. Thus following the argument of \cite[Theorem 1.2]{hein2011weighted}, we can also derive the weighted Sobolev inequality.
Thus our assumption (a)(b) are verified by choosing $\alpha=\frac{n}{n-2},\beta=n$ and $\alpha=1,\beta=n$ in Hein's weighted Sobolev inequality, respectively.
And our curvature decay is derived before. 

We now show that the space of bounded harmonic $1$-forms on $\M$ is infinite-dimensional, which demonstrates the necessity of the smallness assumption on the Ricci curvature. 
Let $\Sigma_k$ be a cross-section of $H_k$, and let $$\Psi_k:[-2b_k,2b_k]\times\Sigma_k\rightarrow H_k$$be the corresponding product parametrization. 
Choose $\chi_k\in C^\infty_0(-b_k,b_k)$ such that $\int_{-b_k}^{b_k}\chi_k(t)dt=1$, and define $\eta_k(x):=\chi_k(t_k)dt_k$ on $H_k$, where $t_k$ denotes the first coordinate with respect to $\Psi_k$.
 Since $\chi_k$ is compactly supported in $(-b_k,b_k)$, we may extend $\eta_k$ by zero to obtain a smooth $1$-form on the whole manifold $\M$. 
 It is then immediate that $\operatorname{supp}\eta_k\subseteq H_k$ and $d\eta_k=\chi_k'(t_k)d t_k\wedge d t_k=0$.

For each $k\geq 1$, we construct a closed curve $\gamma_k$ as follows. 
Starting from a point on one end of the handle $H_k$, the curve passes through this handle and reaches the other end of this handle.
Then it goes back to the starting point in the flat region.
Without loss of generality, we may assume that $\gamma_k$ intersects $\Sigma_k$ transversely at exactly one point. 
Notice that the curve $\gamma_k$ is not null-homotopic for all $k$. Choosing the orientation appropriately, the normalization of $\chi_k$ gives
\begin{equation}
	\int_{\gamma_k}\eta_j=\delta_{jk},\quad j,k\geq1.
\end{equation}

We now show that the classes $[\eta_1],[\eta_2],\cdots,[\eta_N]\in H_c^1(\M)$, are linearly independent in $H_c^1(\M)$, where $H_c^1(\M)$ denotes the de Rham cohomology with compact support. 
Indeed, suppose that, for some $N\geq 1$,\begin{equation}
	c_1[\eta_1]+c_2[\eta_2]+\cdots+c_N[\eta_N]=0,
\end{equation}
 Then there exists a function $f\in C_0^\infty(\M)$ such that \begin{equation}
 	\sum\limits_{j=1}^Nc_j\eta_j=df.
 \end{equation}
Integrating over $\gamma_k$ for each $k\in{1,\ldots,N}$, we obtain $$c_k=\sum\limits_{j=1}^Nc_j\int_{\gamma_k}\eta_j=\int_{\gamma_k}df=0,$$ since $\gamma_k$ is a closed curve. 
Hence $[\eta_1],[\eta_2],\cdots,[\eta_N]$ are linearly independent for every $N$, and therefore $$\dim H_c^1(\M)=\infty.$$

Since $\M$ is complete and has only one end, \cite[Lemma 2.3]{carron20072} implies that the natural map $H_c^1(\M)\rightarrow H_{(2)}^1(\M)$ is injective, where $H_{(2)}^1(\M)$ denotes the reduced $L^2$ cohomology. 
Consequently, $\dim H_{(2)}^1(\M)=\infty$. Moreover, since $\M$ is complete, the reduced $L^2$ cohomology is naturally isomorphic to the space of $L^2$ harmonic forms (see \cite[Corollary 1.6]{carron20072}). 
It follows that the space of $L^2$ harmonic $1$-forms on $\M$ is infinite-dimensional.

On the other hand, by our self-similar construction, $\M$ has bounded geometry. Hence, by standard interior elliptic estimates, for every $L^2$ harmonic $1$-form $\omega$ and every integer $m>\frac n2$, we have \begin{equation}
	||\omega||_{H^m(B_x(\frac{r_0}{2}))}\leq C_m||\omega||_{L^2(B_x(r_0))}\leq C_m||\omega||_{L^2(\M)},
\end{equation}
where $r_0=\min\{1,\frac{1}{4}\operatorname{inj}(\M,g)\}$ is a small radius, and $C_m$ is independent of $x\in\M$. By the Sobolev embedding theorem, \begin{equation}
	\sup\limits_{B_x(\frac{r_0}{2})}|\omega|\leq C_m||\omega||_{L^2(\M)}.
\end{equation} Since $x\in\M$ is arbitrary, this yields\begin{equation}
	||\omega||_{L^\infty(\M)}\leq C_m||\omega||_{L^2(\M)}.
\end{equation}
Thus every $L^2$ harmonic $1$-form on $\M$ is bounded. Since the space of $L^2$ harmonic $1$-forms is infinite-dimensional, the space of bounded harmonic $1$-forms on $\M$ is also infinite-dimensional. This completes the construction of the desired counterexample.

\section{Proof of Theorem \ref{polyrepre}}\label{sec6}
In this section, we will prove Theorem \ref{polyrepre}. First we will prove a lemma.
\begin{lemma}\label{lem5.1}
    Let $\M$ be a complete manifold satisfying the weak Hardy inequality in Theorem \ref{thm1}. 
    And let $u$ be an ancient solution of the equation $\pa_tu-\Delta u-Vu=0$. 
    If the smooth function $V$ satisfies $$ 0\leq V(x)\leq\frac{\varepsilon_0}{1+r^2(x)},$$where $\varepsilon_0$ is sufficiently small such that $\varepsilon_0 C_H<1$. Then there exists a constant $C$ independent of the radius $R>0$ such that \begin{equation}
        R^4\int_{Q_{\frac{R}{10}}}u^2_tdxdt+R^2\int_{Q_{\frac{R}{10}}}|\nabla u|^2dxdt\leq C\int_{Q_R}u^2dxdt
    \end{equation}holds for all $R>0$.
\end{lemma}
\begin{proof}
    For convenience, for any $R>0$, we denote $Q_R:=B_p(R)\times[-R^2,0]$. 
    Let $\phi$ be a cut-off function on $\M$. Then we have \begin{equation}
        \begin{aligned}
            0&=\int_\M \phi^2u\partial_tudx-\phi^2u\Delta udx-Vu^2\phi^2dx\\
            &=\frac{1}{2}\frac{\partial}{\partial t}\int_\M \phi^2u^2dx+\int_\M \langle 2\phi u\nabla\phi+\phi^2\nabla u,\nabla u\rangle dx-\int_\M V\phi^2u^2dx.
        \end{aligned}
    \end{equation}
    Then by applying Cauchy-Schwarz inequality and our assumption (b), we obtain \begin{equation}
    	\begin{aligned}
    		\frac{1}{2}\frac{\pa}{\pa t}&\int_\M\phi^2u^2dx=-\int_\M \langle 2\phi u\nabla\phi+\phi^2\nabla u,\nabla u\rangle dx+\int_\M V\phi^2u^2dx\\
    		&\leq (-1+\delta_2)\int_M\phi^2|\nabla u|^2dx+\frac{1}{\delta_2}\int_\M u^2|\nabla\phi|^2dx+\varepsilon_0\int_\M\frac{\phi^2u^2}{1+r^2(x)}dx\\
    		&\leq \bigg[-1+\delta_2+\varepsilon_0 C_H(1+\eta_1)\bigg]\int_M\phi^2|\nabla u|^2dx+\bigg[\frac{1}{\delta_2}+\varepsilon_0 C_H(1+\frac{1}{\eta_1})\bigg]\int_\M u^2|\nabla\phi|^2dx\\
    	\end{aligned}
    \end{equation}
    Pick $$\delta_2=\frac{1}{2}\big(1-\varepsilon_0 C_H(1+\eta_1)\big),\eta_1=\frac{1}{2}\big(\frac{1}{\varepsilon_0 C_H}-1\big),$$ then we obtain \begin{equation}
        \frac{1}{2}\frac{\pa}{\pa t}\int_\M\phi^2u^2dx+\frac{1}{4}(1-\varepsilon_0 C_H)\int_\M\phi^2|\nabla u|^2dx\leq C\int_\M u^2|\nabla\phi|^2dx.
    \end{equation}
    Choose $\phi$ be a smooth cut-off function such that $$\phi=\bigg\{\begin{aligned}
        &1,\quad x\in B_p(\frac{R}{2})\\
        &0,\quad x\in B_p^C(R)
    \end{aligned},\quad |\nabla\phi|\leq \frac{2}{R}.$$
    Integration on $[-R^2,0]$, then we have \begin{equation}\label{5.4}
        \frac{1}{2}\int_\M \phi^2u^2dx\bigg|_{t=0}+\frac{1}{4}\big(1-\varepsilon_0 C_H\big)\int_{Q_{\frac{R}{2}}}|\nabla u|^2dxdt\leq \frac{C}{R^2}\int_{Q_R}u^2dxdt+\frac{1}{2}\int_{B_p(R)}u^2dx\bigg|_{t=-R^2}.
    \end{equation}
    Now we will derive a similar upper bound for $u_t$. Let $\psi$ be a smooth cut-off function on $\M$. Integration by parts yields \begin{equation}
        \begin{aligned}
            \pa_t\int_\M |\nabla u|^2\psi^2dx&=\int_\M 2\langle \nabla u,\nabla u_t\rangle\psi^2 dx=-\int_\M 2\psi ^2u_t\Delta udx-4\int_\M u_t\psi\langle \nabla u,\nabla\psi\rangle dx\\
            &\leq -\int_\M \psi^2u_t^2dx+\pa_t\int_\M V\psi^2u^2+4\int_\M |\nabla u|^2|\nabla\psi|^2dx.
        \end{aligned}
    \end{equation}
    Choose $\psi$ be a smooth cut-off function such that $$\psi=\bigg\{\begin{aligned}
        &1,\quad x\in B_p(\frac{r_0}{2})\\
        &0,\quad x\in B_p^C(r_0)
    \end{aligned},\quad |\nabla\psi|\leq \frac{2}{r_0}.$$
    Integrate on $[-r_0^2,0]$ and apply Cauchy-Schwarz inequality and weak Hardy inequality, we obtain \begin{equation}\label{5.6}
        \begin{aligned}
            \int_\M \psi^2|\nabla u|^2dx\bigg|_{t=0}&+\int_{Q_{\frac{r_0}{2}}}u_t^2dxdt\\
            &\leq\int_M V\psi^2u^2dx\bigg|_{t=0}+\int_{B_p(r_0)}|\nabla u|^2dx\bigg|_{t=-r_0^2}+\frac{16}{r_0^2}\int_{Q_{r_0}}|\nabla u|^2dxdt\\
            &\leq \varepsilon_0\int_\M\frac{\psi^2u^2}{1+r^2(x)}dx\bigg|_{t=0}+\frac{16}{r_0^2}\int_{Q_{r_0}}|\nabla u|^2dxdt+\int_{B_p(r_0)}|\nabla u|^2dx\bigg|_{t=-r_0^2}\\
            &\leq \varepsilon_0 C_H(1+\eta_2)\int_\M \psi^2|\nabla u|^2dx\bigg|_{t=0}+\varepsilon_0 C_H(1+\frac{1}{\eta_2})\int_\M u^2|\nabla\psi|^2dx\bigg|_{t=0}\\&+\frac{16}{r_0^2}\int_{Q_{r_0}}|\nabla u|^2dxdt+\int_{B_p(r_0)}|\nabla u|^2dx\bigg|_{t=-r_0^2}.
        \end{aligned}
    \end{equation}
    
    Since $\varepsilon_0 C_H<1$, then we can choose $\eta_2$ such that $\varepsilon_0 C_H(1+\eta_2)<1.$ Then we obtain \begin{equation}\label{5.7}
        \int_{Q_{\frac{r_0}{2}}}u_t^2dxdt\leq \varepsilon_0 C_H(1+\frac{1}{\eta_2})\int_\M u^2|\nabla\psi|^2dx\bigg|_{t=0}+\frac{16}{r_0^2}\int_{Q_{r_0}}|\nabla u|^2dxdt+\int_{B_p(r_0)}|\nabla u|^2dx\bigg|_{t=-r_0^2}.
    \end{equation}
    Choose $R_1\in[\frac{4}{5}R,R]$ such that $$\int_{B_p(R)}u^2dx\bigg|_{t=-R_1^2}\leq\frac{25}{9R^2}\int_{Q_R}u^2dxdt.$$ Take $R=R_1$ in (\ref{5.4}), then we have \begin{equation}\label{5.8}
        \frac{1}{2}\int_{B_{\frac{2}{5}R}}u^2dx\bigg|_{t=0}+\frac{1}{4}\big(1-\varepsilon_0 C_H\big
        )\int_{Q_{\frac{2}{5}R}}|\nabla u|^2dxdt\leq \frac{C}{R^2}\int_{Q_R}u^2dxdt.
    \end{equation}
    Similarly, choose $r_1\in[\frac{r_0}{5},\frac{2r_0}{5}]$ such that $$\int_{B_{\frac{2r_0}{5}}}|\nabla u|^2dx\bigg|_{t=-r_1^2}\leq \frac{25}{3r_0^2}\int_{Q_{\frac{2r_0}{5}}}|\nabla u|^2dxdt.$$
    Take $r_0=r_1$ in (\ref{5.7}), we obtain \begin{equation}
        \int_{Q_{\frac{r_0}{10}}}u_t^2dxdt\leq\frac{100\varepsilon_0 C_H}{r_0^2}(1+\frac{1}{\eta_2})\int_{B_{\frac{2r_0}{5}}}u^2dx\bigg|_{t=0}+\frac{C_0}{r_0^2}\int_{Q_{\frac{2r_0}{5}}}|\nabla u|^2dxdt,
    \end{equation} where $C_0$ is a large universal constant.

    Let $R=r_0$ in above estimates, then we have \begin{equation}\label{5.10}
        \begin{aligned}
            \int_{Q_{\frac{R}{10}}}u_t^2dxdt&\leq\frac{100\varepsilon_0 C_H}{R^2} (1+\frac{1}{\eta_2})\int_{B_{\frac{2R}{5}}}u^2dx\bigg|_{t=0}+\frac{C_0}{R^2}\int_{Q_{\frac{2R}{5}}}|\nabla u|^2dxdt\\
            &\leq \frac{100\varepsilon_0 C_H}{R^2}(1+\frac{1}{\eta_2})\int_{B_{\frac{2R}{5}}}u^2dx\bigg|_{t=0}\\&+\frac{C_0}{R^2}\bigg(\frac{C}{R^2}\int_{Q_R}u^2dxdt-\frac{2}{1-\varepsilon_0 C_H}\int_{B_{\frac{2R}{5}}}u^2dx\bigg|_{t=0}\bigg).
        \end{aligned}
    \end{equation}
    Observe that when $C_0>50$, there exists $\eta_2$ such that $100\varepsilon_0 C_H(1+\frac{1}{\eta_2})<\frac{2C_0}{1-\varepsilon_0 C_H}$ and $\varepsilon_0 C_H(1+\eta_2)<1$ both hold. Then by choosing $\eta_2$ as above and combining (\ref{5.8}) and (\ref{5.10}), we obtain \begin{equation}
        R^4\int_{Q_{\frac{R}{10}}}u^2_tdxdt+R^2\int_{Q_{\frac{R}{10}}}|\nabla u|^2dxdt\leq C\int_{Q_R}u^2dxdt.
    \end{equation}
   This completes the proof.
\end{proof}
Using this Lemma, we have the following corollary, which implies Theorem \ref{polyrepre} immediately:
\begin{corollary}\label{cor5.1}
    Let $\M$ be a complete manifold satisfying the assumptions (b) of Theorem~\ref{thm1} and polynomial volume growth, namely, $\Vol(B_p(R))\leq CR^\nu$. Let $u$ be an ancient solution to $\partial_t u - \Delta u - V u = 0$ with polynomial growth of order $l$, where $V$ is the smooth function satisfying same assumptions as in Theorem \ref{polyrepre}, 
then $\partial_t^{\,k} u \equiv 0$ for all integers $k$ satisfying $4k>2l+\nu+2.$
\end{corollary}
\begin{proof}
    Since the metric on $\M$ and function $V$ are time-independent, the operator $\pa_t-\Delta-V$ commutes with $\pa_t$. Consequently, $\pa_t^ju$ is also an ancient solution to $\partial_tu-\Delta u-Vu=0$. Apply Lemma \ref{lem5.1} and by induction, we have \begin{equation}
        \int_{Q_{\frac{R}{10^k}}}|\pa_t^ku|^2dxdt\leq\frac{C(k)}{R^{4k}}\int_{Q_R}u^2dxdt\leq\frac{C}{R^{4k-2-2l}}R^\nu.
    \end{equation}
    If $4k>2l+\nu+2$, then the right-hand side goes to zero when $R\rightarrow\infty$, and this corollary is proved.
\end{proof}

\section{Proof of Theorem \ref{dim_est_anc}}\label{sec7}
In this section we will prove Theorem \ref{dim_est_anc}. This theorem will be obtained from Theorem \ref{thm1} and the following lemma. 
\begin{lemma}\label{lem6.1}
    Let $\M$ be a complete manifold as in Theorem \ref{dim_est_anc} , and let $\mathcal{S}_l(\M),\mathcal{SH}_l(\M)$ be defined as in Theorem \ref{dim_Sch} and \ref{polyrepre}, respectively. Then for every positive integer $k$, \begin{equation}
        \dim \mathcal{SH}_{2k}(\M)\leq (k+1)\dim\mathcal{S}_{2k}(\M).
    \end{equation}
\end{lemma}
\begin{proof}
    First, by Lemma \ref{volume_growth}, our manifold has polynomial volume growth, namely, $Vol(B_p(R))\leq CR^{n+\varepsilon}.$
    Choose an integer $m$ with $4m>4k+n+\varepsilon+2$. By Corollary \ref{cor5.1}, we obtain $\pa_t^mu=0$ for any $u\in\mathcal{SH}_{2k}(\M).$ Then any $u\in\mathcal{SH}_{2k}(\M)$ can be written as \begin{equation}
        u=a_0+a_1t+\cdots+a_{m-1}t^{m-1},
    \end{equation} where each $a_j$ is a function on $\M$. Moreover, since $u\in\mathcal{SH}_{2k}(\M)$, by fixing $x\in\M$ and letting $t\to-\infty$, we see that $a_j=0$ when $j>k$.
    Now we will show that each $a_j$ has a polynomial growth of order $2k$. To prove this, we pick $(k+1)$ distinct values $-1<t_0<t_1<\cdots<t_{k}<0$. And recall the Vandermonde determinant $$\det (t_i^{j})=\prod_{i<j}(t_j-t_i)\neq 0,$$ which yields that the matrix $(t_i^j)_{0\leq i,j\leq k}$ is invertible. Thus we can find $b_i^j$ such that \begin{equation}\label{6.3}
        e_j=\sum\limits_ib_i^j(1,t_i,t_i^2,\cdots,t_i^{k}),
    \end{equation} where $e_j=(0,\cdots,0,1,0,\cdots,0)\in\mathbb{R}^{k+1}$ is the vector in $\mathbb{R}^{k+1}$ with $j$-th component is $1$ and other components are $0$. Then it follows that \begin{equation}
        a_j=\sum\limits_ib_i^ju(x,t_i),
    \end{equation} which implies that each $a_j$ has polynomial growth of order $2k$.
    Since $u$ satisfies the equation $\pa_tu-\De u-Vu=0,$ then by direct computation, we have \begin{equation}
        \begin{aligned}
            &\De a_{k}+Va_{k}=0;\\
            &\De a_{j}+Va_j=(j+1)a_{j+1},\quad 0\leq j\leq k-1.
        \end{aligned}
    \end{equation}
    Consider the linear map $\Psi_0: \mathcal{SH}_{2k}(\M)\rightarrow\mathcal{S}_{2k}(\M)$ given by $\Psi_0(u)=a_k$. Let $\mathcal{K}_0:=\ker \Psi_0.$ Then it follows that $$\dim \mathcal{SH}_{2k}(\M)\leq \dim\mathcal{K}_0+\dim\mathcal{S}_{2k}(\M).$$
    If $u\in\mathcal{K}_0$, then $a_k=0$ and $\De a_{k-1}+Va_{k-1}=0$, so we consider a new linear operator $\Psi_1:\mathcal{K}_0\rightarrow \mathcal{S}_{2k}(\M)$ given by $\Psi_1(u)=a_{k-1}$. Let $\mathcal{K}_1=\ker \Psi_1$, then we have $$\dim\mathcal{K}_0\leq \dim\mathcal{K}_1+\dim\mathcal{S}_{2k}(\M).$$
    Repeating this process for $k+1$ times gives this result.
\end{proof}
Though Lemma \ref{lem6.1} is enough to prove our Theorem \ref{dim_est_anc}, this lemma is still a rough estimate. And our next lemma will help improve this estimate.
\begin{lemma}\label{lem6.2}
    If $u\in \mathcal{SH}_{2k}(\M)$ can be written as a polynomial of time $t$:$$u=a_0+a_1t+\cdots+a_kt^k,$$ then we have \begin{equation}
        |a_j(x)|\leq C_j(1+r(x)^{2k-2j}).
    \end{equation}
\end{lemma}
\begin{proof}
    By growth assumption on $u$, there is a constant $C$ such that \begin{equation}
        |u(x,t)|\leq C(1+r(x)^{2k}+|t|^k).
    \end{equation}
    Fix $-1<t_0<t_1<\cdots<t_k<-\frac{1}{2}$ and coefficients $b_i^j$ so that (\ref{6.3}) holds for each $j$. Notice that \begin{equation}
        \sum\limits_ib_i^ju(x,R^2t_i)=\sum\limits_i\sum\limits_lb_i^ja_lR^{2l}t_i^l=\sum\limits_l\sum\limits_ib_i^ja_lR^{2l}t_i^l=R^{2j}a_j(x),
    \end{equation}
    thus we have \begin{equation}
       \begin{aligned}
            |R^{2j}a_j|&=|\sum\limits_ib_i^ju(x,R^2t_i)|\leq \max\limits_{i,j}|b_i^j|\sum\limits_{i}|u(x,R^2t_i)|\\
            &\leq C_1(1+r(x)^{2k}+\max\limits_{i}|R^2t_i|^k)\leq 3C_1R^{2k}.
       \end{aligned}
    \end{equation}
    From this we obtain our lemma by taking $R=1+r(x)$.
\end{proof}
Now we improve our estimate of Lemma \ref{lem6.1} by using Lemma \ref{lem6.2}.
\begin{theorem}
    Let $\M$ be a complete manifold as in Theorem \ref{dim_est_anc}, and let $\mathcal{SH}_l(\M),\mathcal{S}_l(\M)$ be defined as above. Then for any positive integer $k$, we have \begin{equation}
        \dim\mathcal{SH}_{2k}(\M)\leq \sum\limits_{i=0}^k\dim \mathcal{S}_{2i}(\M).
    \end{equation}
\end{theorem}
\begin{proof}
    Following the proof of Lemma \ref{lem6.1}, each $u\in\mathcal{SH}_{2k}(\M)$ can be written as $$u=a_0+a_1t+\cdots a_kt^k.$$
    By Lemma \ref{lem6.2}, the linear map $\Psi_0:\mathcal{SH}_{2k}(\M)\rightarrow\mathcal{S}_{2k}(\M)$ given by $\Psi_0(u)=a_k$ actually maps into $\mathcal{S}_0(\M)$, and thus \begin{equation}
        \dim\mathcal{SH}_{2k}(\M)\leq \dim\mathcal{S}_0(\M)+\dim\ker\Psi_0.
    \end{equation}
    Similarly, Lemma \ref{lem6.2} implies that the map $\Psi_1$ actually maps $\ker \Psi_0$ into $\mathcal{S}_2(\M)$. Repeating this process for $k+1$ times yields this theorem.
\end{proof}
\section{Appendix}
In this appendix, we present a lemma, which allows us to bound the integral over the geodesic ball by the integral over the dyadic annuli. 
Since the lemma plays a useful role in our paper, its proof is presented for completeness.
This result is communicated to us by Qi S. Zhang, Department of Mathematics, UC Riverside, CA 92521. 

\begin{lemma}\cite{zq}\label{lem_appendix}
    Let $\M$ be a complete manifold satisfying the assumptions of Theorem \ref{thm1}. Then for any $R>0$, any $\lambda>0$ and any harmonic 1-form $\omega$, there exists a constant $C$ depending only on $C_H,C_S,\varepsilon$ such that \begin{equation}\label{2.3}
        \int_{B_p(R)}|\omega|^2dx\leq C\frac{(1+\frac{1}{2}\lambda)^2}{\lambda^2}\int_{B_p((1+\lambda)R)\backslash B_p(R)}|\omega|^2dx.
    \end{equation} 
\end{lemma}
\begin{proof}
    By Bochner-Weitzenb\"ock formula, we have $\Delta |\omega|^2=2|\nabla\omega|^2+2\Ric(\omega,\omega).$ Let $\phi$ be a smooth cut-off function such that \begin{equation}
        \phi=\bigg\{\begin{aligned}
            &1,\quad B_p((1+\frac{1}{2}\lambda)R)\\
            &0,\quad B^C_p((1+\lambda)R),
        \end{aligned}
        \quad |\nabla\phi|\leq\frac{4}{\lambda R}.
    \end{equation}
    Integration by parts yields \begin{equation}
        \begin{aligned}
           \int_\M2\phi^2|\nabla\omega|^2dx+\int_\M2\phi^2\Ric(\omega,\omega)dx&=\int_\M\phi^2\Delta|\omega|^2dx\\
           &=-\int_\M\langle2\phi\nabla\phi,2|\omega|\nabla|\omega|\rangle dx.
        \end{aligned}
    \end{equation}
    By applying Cauchy-Schwarz inequality, we obtain that \begin{equation}\label{eq2.6}
        \int_\M\phi^2|\nabla\omega|^2dx\leq\varepsilon'\int_\M\phi^2|\nabla\omega|^2dx+\frac{1}{\varepsilon'}\int_\M|\omega|^2|\nabla\phi|^2dx+\int_\M\phi^2|\Ric_-||\omega|^2dx.
    \end{equation} 
    By applying our Hardy type inequality and Cauchy-Schwarz inequality, we obtain that \begin{equation}
        \begin{aligned}
            \int_\M\phi^2|\Ric_-||\omega|^2dx&\leq \int_\M\varepsilon\frac{\phi^2|\omega|^2}{1+r^2(x)}dx\leq \varepsilon C_H\int_\M|\nabla(\phi|\omega|)|^2dx\\
            &\leq\varepsilon C_H(1+\delta)\int_\M\phi^2|\nabla\omega|^2dx+\varepsilon C_H(1+\frac{1}{\delta})\int_\M|\omega|^2|\nabla\phi|^2dx
        \end{aligned}
    \end{equation} where we used the Kato inequality in the last inequality. 
By our assumption on $\varepsilon$, we can choose $\varepsilon',\delta$ so that $$(1+\delta)\varepsilon C_H+\varepsilon'<1.$$
Thus we obtain that \begin{equation}
    \int_\M\phi^2|\nabla\omega|^2dx\leq C_1\int_\M|\omega|^2|\nabla\phi|^2dx,
\end{equation} where $C_1$ is a constant depending only on $\varepsilon$ and $C_H$. By our choice of cut-off function $\phi$, we have \begin{equation}
        \int_{B_p((1+\frac{1}{2}\lambda)R)}|\nabla\omega|^2dx\leq\frac{16C_1}{\lambda^2R^2}\int_{B_p((1+\lambda)R)\backslash B_p((1+\frac{1}{2}\lambda)R)}|\omega|^2dx
    \end{equation}
    Now we take another smooth cut-off function $\psi$ such that \begin{equation}
        \psi=\bigg\{\begin{aligned}
            &1,\quad B_p(R)\\
            &0,\quad B^C_p((1+\frac{1}{2}\lambda)R)
        \end{aligned},\quad |\nabla\psi|\leq\frac{4}{\lambda R}.
    \end{equation}
    Thus, we have \begin{equation}
        \begin{aligned} \int_{B_p((1+\frac{1}{2}\lambda)R)}|\nabla(\psi\omega)|^2dx&=\int_{B_p((1+\frac{1}{2}\lambda)R)}|\nabla\psi\otimes\omega+\psi\nabla\omega|^2dx\\
        &\leq 2\int_{B_p((1+\frac{1}{2}\lambda)R)}|\nabla\psi|^2|\omega|^2dx+2\int_{B_p((1+\frac{1}{2}\lambda)R)}|\psi|^2|\nabla\omega|^2dx\\
        &\leq \frac{32}{\lambda^2R^2}\int_{B_p((1+\frac{1}{2}\lambda)R)\backslash B_p(R)}|\omega|^2dx+\frac{32C_1}{\lambda^2R^2}\int_{B_p((1+\lambda)R)\backslash B_p((1+\frac{1}{2}\lambda)R)}|\omega|^2dx\\
        &\leq \frac{32(1+C_1)}{\lambda^2R^2}\int_{B_p((1+\lambda)R)\backslash B_p(R)}|\omega|^2dx.
        \end{aligned}
    \end{equation}
    On the other hand, by applying H\"older inequality and Sobolev inequality, we establish the following inequalities:\begin{equation}
        \begin{aligned}
            \int_{B_p((1+\frac{1}{2}\lambda)R)}|\nabla(\psi\omega)|^2dx&\geq\int_{B_p((1+\frac{1}{2}\lambda)R)}|\nabla(|\psi\omega|)|^2dx\\
            &\geq \frac{|B_p((1+\frac{1}{2}\lambda)R)|}{C_S(1+\frac{1}{2}\lambda)^2R^2}\bigg(\fint_{B_p((1+\frac{1}{2}\lambda)R)}|\psi\omega|^{\frac{2n}{n-2}}dx\bigg)^{\frac{n-2}{n}}\\
            &\geq \frac{|B_p((1+\frac{1}{2}\lambda)R)|}{C_S(1+\frac{1}{2}\lambda)^2R^2}\fint_{B_p((1+\frac{1}{2}\lambda)R)}|\psi\omega|^2dx\\
            &\geq \frac{1}{C_S(1+\frac{1}{2}\lambda)^2R^2}\int_{B_p(R)}|\omega|^2dx
        \end{aligned}
    \end{equation}
    where in the last inequality we use the fact that $\psi\equiv1$ on $B_p(R)$.  Thus our lemma follows by combining the above two estimates. 
\end{proof}
Using a similar argument, we can also establish the stationary Schr\"odinger version of this lemma, which states as follows:
\begin{lemma}
	Let $\M$ be a complete manifold satisfying the assumptions (a)(b) of Theorem \ref{thm1}. 
    And let $V$ be a smooth function on $\M$ satisfying $$V(x)\leq \frac{\varepsilon_0}{1+r^2(x)},$$ where $\varepsilon_0$ is a sufficiently small number such that $\varepsilon_0 C_H<1$. 
    Then for any $R>0$, any $\lambda>0$ and any smooth solution of the stationary Schr\"odinger equation $\Delta u+Vu=0$, there exists a constant $C$ depending only on $C_H,C_S,\varepsilon_0$ such that 
    \begin{equation}\label{2.4}
        \int_{B_p(R)}u^2dx\leq C\frac{(1+\frac{1}{2}\lambda)^2}{\lambda^2}\int_{B_p((1+\lambda)R)\backslash B_p(R)}u^2dx.
    \end{equation} 
\end{lemma}

\section*{Acknowledgements}
The author would like to thank Prof. Qi S. Zhang for many valuable discussions and suggestions. 
The author is also grateful to his supervisor Prof. Bobo Hua for helpful discussions. 
Finally the author would like to thank Prof. Jiaping Wang for his interest in this work and valuable feedback.

\section*{AI assistance statement}
The author used OpenAI’s ChatGPT for language polishing and for assistance in checking some computational details in Section \ref{sec5}. 
All mathematical arguments, proofs, and conclusions were independently verified by the author, who takes full responsibility for the content of this paper.
\bibliographystyle{alpha}
\bibliography{ref}

@article{lin2019ancient,
  title={On ancient solutions of the heat equation},
  author={Lin, Fanghua and Zhang, Qi S},
  journal={Communications on Pure and Applied Mathematics},
  volume={72},
  number={9},
  pages={2006--2028},
  year={2019},
  publisher={Wiley Online Library}
}

@article{colding2021optimal,
  title={Optimal bounds for ancient caloric functions},
  author={Colding, Tobias Holck and Minicozzi II, William P},
  journal={Duke Mathematical Journal},
  volume={170},
  number={18},
  pages={4171--4182},
  year={2021},
  publisher={Duke University Press}
}

@article{sung2022polynomial,
  title={Polynomial growth ancient solutions to harmonic form heat flow},
  author={Sung, Chiung-Jue Anna and Tsai, Yi-Hao},
  journal={Archiv der Mathematik},
  volume={118},
  number={6},
  pages={637--647},
  year={2022},
  publisher={Springer}
}

@article{colding1997harmonic,
  title={Harmonic functions on manifolds},
  author={Colding, Tobias H and Minicozzi, William P},
  journal={Annals of mathematics},
  volume={146},
  number={3},
  pages={725--747},
  year={1997},
  publisher={JSTOR}
}

@article{li1997harmonic,
  title={Harmonic sections of polynomial growth},
  author={Li, Peter},
  journal={Mathematical Research Letters},
  volume={4},
  number={1},
  pages={35--44},
  year={1997},
  publisher={International Press of Boston}
}

@article {chen2006dimension,
    AUTHOR = {Chen, Jui-Tang Ray and Sung, Chiung-Jue Anna},
     TITLE = {Dimension estimate of polynomial growth harmonic forms},
   JOURNAL = {J. Differential Geom.},
  FJOURNAL = {Journal of Differential Geometry},
    VOLUME = {73},
      YEAR = {2006},
    NUMBER = {1},
     PAGES = {167--183},
      ISSN = {0022-040X,1945-743X},
   MRCLASS = {53C65 (58A10 58J50)},
  MRNUMBER = {2217522},
MRREVIEWER = {Man\ Chun\ Leung},
       URL = {http://projecteuclid.org/euclid.jdg/1146680515},
}

@article{li2012geometric,
  title={Geometric analysis},
  author={Li, P},
  journal={Cambridge Studies in Advanced Mathematics/Cam-bridge University Press},
  volume={134},
  year={2012}
}

@article{colding2020complexity,
  title={Complexity of parabolic systems},
  author={Colding, Tobias Holck and Minicozzi, William P},
  journal={Publications math{\'e}matiques de l'IH{\'E}S},
  volume={132},
  number={1},
  pages={83--135},
  year={2020},
  publisher={Springer}
}

@article{ni2004monotonicity,
  title={A monotonicity formula on complete {K{\"a}hler} manifolds with nonnegative bisectional curvature},
  author={Ni, Lei},
  journal={Journal of the American Mathematical Society},
  volume={17},
  number={4},
  pages={909--946},
  year={2004}
}

@article{chen2006sharp,
  title={Sharp dimension estimates of holomorphic functions and rigidity},
  author={Chen, Bing-Long and Fu, Xiao-Yong and Yin, Le and Zhu, Xi-Ping},
  journal={Transactions of the American Mathematical Society},
  volume={358},
  number={4},
  pages={1435--1454},
  year={2006}
}

@article{hua2011harmonic,
  title={Harmonic functions of polynomial growth on singular spaces with nonnegative {Ricci} curvature},
  author={Hua, Bobo},
  journal={Proceedings of the American Mathematical Society},
  volume={139},
  number={6},
  pages={2191--2205},
  year={2011}
}

@article{ho2023ancient,
  title={Ancient Caloric Functions on Pseudohermitian Manifolds},
  author={Ho, Pak Tung},
  journal={The Journal of Geometric Analysis},
  volume={33},
  number={2},
  pages={41},
  year={2023},
  publisher={Springer}
}

@article{liu2016three,
  title={Three-circle theorem and dimension estimate for holomorphic functions on {K{\"a}hler} manifolds},
  author={Liu, Gang},
  journal={Duke Mathematical Journal},
  volume={165},
  number={15},
  pages={2899--2919},
  year={2016},
  publisher={Duke University Press}
}

@article{liu2021dimension,
  title={Dimension estimate of polynomial growth holomorphic functions},
  author={Liu, Gang},
  journal={Peking Mathematical Journal},
  volume={4},
  number={2},
  pages={187--202},
  year={2021},
  publisher={Springer}
}

@article{liu2023brief,
  title={A brief survey on dimension estimate of holomorphic functions on {K{\"a}hler} manifolds},
  author={Liu, Gang},
  journal={The Journal of Geometric Analysis},
  volume={33},
  number={4},
  pages={118},
  year={2023},
  publisher={Springer}
}

@article{he2024dimension,
  title={The dimension of polynomial growth holomorphic functions and forms on gradient {K\"ahler Ricci} shrinkers},
  author={He, Fei and Ou, Jianyu},
  journal={arXiv preprint arXiv:2401.02685},
  year={2024}
}

@inproceedings{li1980sobolev,
  title={On the {Sobolev} constant and the $ p $-spectrum of a compact {Riemannian} manifold},
  author={Li, Peter},
  booktitle={Annales scientifiques de l'{\'E}cole Normale Sup{\'e}rieure},
  volume={13},
  number={4},
  pages={451--468},
  year={1980}
}

@article{wu2023harmonic,
  title={Harmonic and {Schr{\"o}dinger} functions of polynomial growth on gradient shrinking {Ricci} solitons},
  author={Wu, Jia-Yong and Wu, Peng},
  journal={Geometriae Dedicata},
  volume={217},
  number={4},
  pages={75},
  year={2023},
  publisher={Springer}
}

@book{han2011elliptic,
  title={Elliptic partial differential equations},
  author={Han, Qing and Lin, Fanghua},
  volume={1},
  year={2011},
  publisher={American Mathematical Soc.}
}

@phdthesis{calle2007mean,
  title={Mean curvature flow and minimal surfaces},
  author={Calle, Maria},
  year={2007},
  school={New York University}
}

@article{calle2006bounding,
  title={Bounding dimension of ambient space by density for mean curvature flow},
  author={Calle, Maria},
  journal={Mathematische Zeitschrift},
  volume={252},
  pages={655--668},
  year={2006},
  publisher={Springer}
}

@article{kleiner2010new,
  title={A new proof of {Gromov’s} theorem on groups of polynomial growth},
  author={Kleiner, Bruce},
  journal={Journal of the American Mathematical Society},
  volume={23},
  number={3},
  pages={815--829},
  year={2010}
}

@article{gromov1981groups,
  title={Groups of polynomial growth and expanding maps (with an appendix by Jacques Tits)},
  author={Gromov, Michael},
  journal={Publications Math{\'e}matiques de l'IH{\'E}S},
  volume={53},
  pages={53--78},
  year={1981}
}

@article{donnelly2001harmonic,
  title={Harmonic functions on manifolds of nonnegative {Ricci} curvature},
  author={Donnelly, Harold},
  journal={International Mathematics Research Notices},
  volume={2001},
  number={8},
  pages={429--434},
  year={2001},
  publisher={OUP}
}

@book{saloff2002aspects,
  title={Aspects of {Sobolev-type} inequalities},
  author={Saloff-Coste, Laurent},
  volume={289},
  year={2002},
  publisher={Cambridge University Press}
}

@article{carron1997inegalites,
  title={In{\'e}galit{\'e}s de {Hardy} sur les vari{\'e}t{\'e}s riemanniennes non-compactes},
  author={Carron, Gilles},
  journal={Journal de math{\'e}matiques pures et appliqu{\'e}es},
  volume={76},
  number={10},
  pages={883--891},
  year={1997},
  publisher={Elsevier Masson}
}

@article{kombe2009improved,
  title={Improved {Hardy} and {Rellich} inequalities on {Riemannian} manifolds},
  author={Kombe, Ismail and {\"O}zaydin, Murad},
  journal={Transactions of the American Mathematical Society},
  volume={361},
  number={12},
  pages={6191--6203},
  year={2009}
}

@article{lin2026space,
  title={Space of Ancient Caloric Functions on Some Manifolds Beyond Volume Doubling},
  author={Lin, Fanghua and Qiu, Hongbing and Sun, Jun and Zhang, Qi S},
  journal={The Journal of Geometric Analysis},
  volume={36},
  number={4},
  pages={128},
  year={2026},
  publisher={Springer}
}

@article{wolff1992property,
  title={A property of measures in RN and an application to unique continuation},
  author={Wolff, Thomas H},
  journal={Geometric \& Functional Analysis GAFA},
  volume={2},
  number={2},
  pages={225--284},
  year={1992},
  publisher={Springer}
}

@article{colding1998liouville,
  title={Liouville theorems for harmonic sections and applications},
  author={Colding, Tobias H and Minicozzi, William P},
  journal={Communications on Pure and Applied Mathematics: A Journal Issued by the Courant Institute of Mathematical Sciences},
  volume={51},
  number={2},
  pages={113--138},
  year={1998},
  publisher={Wiley Online Library}
}

@article{generau2020laplacian,
  title={Laplacian of the distance function on the cut locus on a {Riemannian} manifold},
  author={G{\'e}n{\'e}rau, Fran{\c{c}}ois},
  journal={Nonlinearity},
  volume={33},
  number={8},
  pages={3928},
  year={2020},
  publisher={IOP Publishing}
}

@article{he2025dimension,
  title={Dimension Estimate and Existence of Holomorphic Sections With Polynomial Growth on Gradient {K{\"a}hler} Ricci Shrinkers},
  author={He, Fei and Ou, Jianyu},
  journal={International Mathematics Research Notices},
  volume={2025},
  number={23},
  pages={rnaf350},
  year={2025},
  publisher={Oxford University Press}
}

@article{tam1998note,
  title={A note on harmonic forms on complete manifolds},
  author={Tam, Luen-Fai},
  journal={Proceedings of the American Mathematical Society},
  volume={126},
  number={10},
  pages={3097--3108},
  year={1998}
}

@article{choi1984asymptotic,
  title={Asymptotic Dirichlet problems for harmonic functions on {Riemannian} manifolds},
  author={Choi, Hyeong In},
  journal={Transactions of the American Mathematical Society},
  volume={281},
  number={2},
  pages={691--716},
  year={1984}
}

@article{colding1996generalized,
author = {Tobias H. Colding and William P. Minicozzi},
title = {{Generalized Liouville properties of manifolds}},
volume = {3},
journal = {Mathematical Research Letters},
number = {6},
publisher = {International Press of Boston},
pages = {723 -- 729},
year = {1996},
doi = {10.4310/MRL.1996.v3.n6.a1},
URL = {https://doi.org/10.4310/MRL.1996.v3.n6.a1}
}

@article{zq,
    author = {Qi S. Zhang},
    title = {private communication}
}

@article{li2019new,
  title={A new complete {Calabi--Yau} metric on $\mathbb{C}^3$},
  author={Li, Yang},
  journal={Inventiones mathematicae},
  volume={217},
  number={1},
  pages={1--34},
  year={2019},
  publisher={Springer}
}

@article{hein2011weighted,
  title={Weighted {Sobolev} inequalities under lower {Ricci} curvature bounds},
  author={Hein, Hans-Joachim},
  journal={Proceedings of the American Mathematical Society},
  volume={139},
  number={8},
  pages={2943--2955},
  year={2011}
}

@article{carron20072,
  title={L\^{}2 harmonics forms on non compact manifolds},
  author={Carron, Gilles},
  journal={arXiv preprint arXiv:0704.3194},
  year={2007}
}
\end{document}